\documentclass[11pt]{amsart}
\usepackage[dvipsnames]{xcolor}
\usepackage{a4wide, amsmath, amsthm, amscd, amsfonts, amssymb, caption, enumerate, epsfig, float, graphics, graphicx, hyperref, mathrsfs, subcaption, textcomp, tikz, tikz-cd, todonotes, verbatim, xcolor, ytableau,pinlabel,subcaption}
\usepackage[T1]{fontenc}

\usepackage{subfiles}

\newtheorem{theorem}{Theorem}[section]
\newtheorem{proposition}[theorem]{Proposition}
\newtheorem{conjecture}[theorem]{Conjecture}
\newtheorem{corollary}[theorem]{Corollary}
\newtheorem{lemma}[theorem]{Lemma}
\newtheorem{question}[theorem]{Question}
\newtheorem{definition}[theorem]{Definition}

\theoremstyle{definition}

\newtheorem{remark}[theorem]{Remark}
\newtheorem{example}[theorem]{Example}

\definecolor{darkgreen}{RGB}{0,153,0}
\definecolor{darkred}{RGB}{204,0,0}
\definecolor{darkblue}{RGB}{0,51,204}
\definecolor{red}{RGB}{242,43,29}

\hypersetup{
  colorlinks,
  citecolor=Purple,
  linkcolor=Blue,
  urlcolor=Green}
    
\newcommand{\Z}{\mathbb{Z}}

\newcommand{\R}{\mathbb{R}}
\newcommand{\C}{\mathbb{C}} 
\newcommand{\F}{\mathbb{F}} 
\newcommand{\A}{\mathcal{A}}

\DeclareMathOperator{\Aug}{Aug}

\newcommand{\e}{\varepsilon}
\newcommand{\La}{\Lambda}

\newcommand{\sse}{\subseteq}
\newcommand{\dd}{\partial}

\begin{document}

\title{On Ruling polynomials of Legendrian links}
\author[O. Capovilla-Searle]{Orsola Capovilla-Searle}
\address{} \email{capovilo@oregonstate.edu}
\author[Y. Pan]{Yu Pan}
\address{}
\email{ypan@tju.edu.cn}

\begin{abstract}
The ruling polynomial is a Legendrian invariant that is closely related to the augmentation variety of the Legendrian.
We characterize all graded and ungraded ruling polynomials, and construct Legendrian links realizing each possible polynomial. The graded augmentation varieties of the Legendrians we construct all have trivial cluster algebra structures. Finally, we construct Legendrian knots admitting $k$ exact Lagrangian fillings with $\chi(L)=n$
that are pairwise smoothly non-isotopic for $n\leq 1$, and $k\geq 0.$
\end{abstract}

\maketitle

\section{Introduction}

\subsection{Summary}
We characterize the polynomials that can be realized as the graded and ungraded ruling polynomial of a Legendrian in $\R^3_{std}$ and construct explicit Legendrian links for each possible ruling polynomial. The family of Legendrian links we construct satisfies additional properties of interest: their augmentation varieties have trivial cluster algebra structures, and they admit families of exact Lagrangian fillings that are pairwise homeomorphic and smoothly non-isotopic. As a consequence there are Legendrians with the same graded ruling polynomial and distinct graded augmentation varieties.  We caution the readers that the augmentation variety in this paper is the quotient of the naive augmentation variety by dg-homotopy. It is not in general an algebraic variety but for all our examples it is.

We develop a tool which may be of independent interest that relies on a gauge-theoretic obstruction to detect when two pairs of smooth immersed ribbon fillings with transverse double points are not smoothly isotopic rel boundary. We construct Legendrian knots with $k$ embedded exact Lagrangian fillings with $\chi(L)=n$ that are pairwise smoothly non-isotopic for $n\leq 1$, $k\geq 0$. We also construct transverse knots with $k$ embedded orientable symplectic fillings with genus $g$ that are pairwise smoothly non-isotopic for $k,g\geq 0.$

\subsection{Context}
Legendrian links in $\R^3$ with the standard contact structure are central objects of study in contact topology~\cite{Arnold_1990, Geiges_2008,EN_survey}. In addition to the classical Legendrian invariants (the Thurston Bennequin number and rotation number), there are Legendrian invariants arising from Floer theory, generating families, and microlocal sheaf theory~\cite{Chekanov, SFT, Traynor_2001, Fuchs_Rutherford_2011,Shende_Treumann_Zaslow_2016}. We focus on graded normal rulings, defined independently by~\cite{Fuchs03,ChekanovPushkar}. A graded normal ruling of a Legendrian $\Lambda\subset \R^3_{std}$ is a decomposition of its Legendrian front $D\subset \R^2$ into a collection of embedded disks subject to certain conditions (see Definition~\ref{defn:ruling}). The graded ruling polynomial, $R_{\La}(z)$, of a Legendrian $\La$ with a fixed front $D$ encodes counts of graded normal rulings of the front $D$:
$$R_{\La}(z)= z^{|\pi_0 (\La)|} \sum_{\mbox{ruling }\rho} z^{-\chi(\rho)},$$
where $|\pi_0 (\La)|$ is the number of connected components of $\La$ and $\chi(\rho)$ is the Euler characteristic of the ruling $\rho$. 

The graded ruling polynomial is related to  smooth knot invariants such as the HOMFLY-PT polynomial~\cite{Ru}, and Legendrian knot invariants such as the Chekanov-Eliashberg differential graded algebra, also known as the Legendrian contact dg-algebra. Defined independently by Chekanov~\cite{Chekanov} and Eliashberg-Givental-Hofer~\cite{SFT} as part of Symplectic Field Theory, the Legendrian contact dg-algebra is a Legendrian invariant up to stable tame isomorphism. Graded normal rulings are in a non-bijective correspondence with the set of augmentations of the Legendrian contact dg-algebra~\cite{Fuchs03, FI,Sab}. 
 We define the augmentation variety to be the quotient of the naive augmentation variety (defined over \(\mathbb{C}\)) modulo dg-homotopy equivalence. This quotient space is generally not an algebraic variety, but we keep the convention to match the current literature despite this clear shortcoming. Every Legendrian we study in this paper has an augmentation variety that is an honest algebraic variety. See Remark~\ref{rem:aug_stack} for more details.
The augmentation variety is an important tool in the study of exact Lagrangian fillings of Legendrian links in $(\mathbb{B}^4, \omega_{std})$, as embedded and immersed fillings induce augmentations~\cite{EHK,Pan_Rutherford_2021,Pan_Rutherford_2023}. The classification of exact Lagrangian fillings of Legendrian links in $(\mathbb{B}^4, \omega_{std})$ up to Hamiltonian (exact Lagrangian) isotopy fixing the boundary is an important problem in contact and symplectic geometry~\cite{eliashberg_polterovich, Chantraine_2010, EHK, Pan_2017}. Connections between the moduli space of fillings and cluster algebras have generated exciting progress in both fields~\cite{STWZ,CG,Casals_weng,Gao2,CSHW,CGGILS}. Thanks to the correspondence between augmentations and graded normal rulings, the graded ruling polynomial of $\La$ encodes finite point counts of the naive augmentation variety~\cite{HR, NRSS}. There are also newfound relations between graded ruling polynomials and BPS invariants arising in the enumerative geometry of planar curve singularities~\cite{SuXiYu}.

The various connections between graded ruling polynomials and Legendrian invariants, as well as geometric objects such as exact Lagrangian fillings and augmentation varieties motivate the following geography question:

\begin{question}\label{ques:geography}
Which polynomials can be realized as the ruling polynomial of a Legendrian knot in $\R^3_{std}$?
\end{question}

We answer this question in Theorem~\ref{thm:main} for graded ruling polynomials and in Theorem~\ref{thm:non-orientable} for ungraded ruling polynomials, a generalization of graded ruling polynomials which can encode exact non-orientable fillings. 
Next, motivated by computational evidence of a relation between cluster structures on the augmentation variety $\A ug(\La)$ and the graded ruling polynomial $R_{\La}(z)$ (see section~\ref{sec:motivation}) we ask and answer the following question affirmatively: 

\begin{question}\label{ques:aug_geography}
Are there Legendrian links in $\R^3_{std}$ with the same ruling polynomial and  distinct augmentation varieties?
\end{question}

Finally, we ask and answer a geography question for exact Lagrangian fillings of Legendrian links. There exist Legendrian knots with an arbitrary number of topologically distinct fillings~\cite{Cao_Gallup_Hayden_Sabloff_2014}, as well as Legendrian knots with even numbers of smoothly non-isotopic exact Lagrangian disk fillings~\cite{Li_Tange_2020}. In another direction, there are Legendrians with infinitely many distinct exact Lagrangian fillings of genus $g>0$ that are all smoothly isotopic rel boundary~\cite{CG}.

\begin{question}\label{ques:fillings}
Let $n\leq 1$ and $k\geq 0$ be arbitrary integers. Are there Legendrian knots in $\R^3_{std}$ with at least $k$ smoothly non-isotopic exact Lagrangian fillings with $\chi(L)=n$?
\end{question}

\subsection{Realizing all possible graded and ungraded ruling polynomials}\label{sec:ruling}

For a graded ruling polynomial, the power of each term, $|\pi_0 (\La)| - \chi(\rho)$, is an even number. 
This is because the graded condition implies that each graded normal ruling $\rho$ can be viewed as an oriented filling of $\La$.
Hence we have  $|\pi_0(\La)| + \chi(\rho)=2|\pi_0( \rho)|-2 g_{\rho}$ where $g_\rho$ is the sum of the genus of all the connected components of $\rho$, it follows that $|\pi_0(\La)| - \chi(\rho)=2|\pi_0(\La)|-2|\pi_0( \rho)|-2 g_{\rho}\in 2\mathbb{Z}.$ Thus, any graded ruling polynomial must be a polynomial $P(z^2)$ with positive integer coefficients. We answer Question~\ref{ques:geography} for the graded case:

\begin{theorem}\label{thm:main}
    Any polynomial $P(z^2)=\displaystyle{\sum^n_{k=0} a_k z^{2k}}$ with non-negative integer coefficients can be realized as the graded ruling polynomial of a Legendrian link $\La_{P(z^2)}$ in $\R^3_{std}$. 
\end{theorem}

We prove Theorem~\ref{thm:main} by explicitly constructing a Legendrian $\La_{P(z^2)}$ with graded ruling polynomial $P(z^2)$.  The Legendrians $\La_{P(z^2)}$ all have the following property:

\begin{theorem}\label{thm:aug}
Let $P(z^2)=\displaystyle{\sum^n_{k=0} a_k z^{2k}}$ be a polynomial with non-negative coefficients. The augmentation variety of the Legendrian link $\La_{P(z^2)}$ is a disjoint union of $a_k$ copies of $(\C^*)^{2k}$ over all values of $k$.
\end{theorem}

 Note that the augmentation variety of $\La_{P(z^2)}$ is an honest variety.

\begin{remark}
Theorem~\ref{thm:main} and Theorem~\ref{thm:aug} hold for Legendrian knots. There exists a sequence of SZ-pinch moves (see Section~\ref{sec:SZ}) that relate $\La_{P(z^2)}$ to a Legendrian knot with the same graded ruling polynomial and an isomorphic augmentation variety.
\end{remark}

As a corollary of Theorem~\ref{thm:aug} we answer Question~\ref{ques:aug_geography} affirmatively.
\begin{corollary}\label{cor:diff_aug}
    There exist Legendrian links with the same graded ruling polynomial but non-homeomorphic augmentation varieties.
\end{corollary}
In particular, consider the pair of the max-tb trefoil $\La_{3_1}$ and $\La_{z^2+2}$. The max-tb trefoil $\La_{3_1}$ has graded ruling polynomial equal to $z^2+2$, and its augmentation variety is the braid variety $\{(x,y,z)\in \C^3~|~x+z+xyz=1\}$ which has an $A_2$-cluster structure. The augmentation variety of $\La_{z^2+2}$ is homeomorphic to the disjoint union of $(\C^*)^2$ and two isolated points. Any variety given by $\sqcup_i^n(\C^*)^{k_i}$ has a trivial cluster algebra. See Section~\ref{sec:motivation} for additional context on Question~\ref{ques:aug_geography}.

\begin{remark}
Question~\ref{ques:aug_geography} can be posed for the moduli of pseudo-perfect objects in the smooth dg-category of sheaves
with singular support on a Legendrian link. However, the sheaf-theoretic computations for our family of Legendrians are much more complicated since they do not admit binary Maslov potentials.
\end{remark}

The ungraded polynomial counts ``ungraded" rulings (see Definitions~\ref{defn:ruling} and~\ref{defn:ruling_polynomial}).
In contrast to the case of graded ruling polynomials, there are restrictions on which polynomials are ungraded ruling polynomials arising from the fact that the Euler characteristic of any two immersed exact Lagrangian fillings of a Legendrian must be equal mod 2, see Lemma~\ref{prop:parity}.

\begin{theorem}\label{thm:non-orientable}
Only the polynomials of the form $P(z^2)$ or $zP(z^2)$ can be realized as the ungraded ruling polynomial of a Legendrian link in $\R^3_{std}$.
\end{theorem}

\noindent An immediate corollary of Theorem~\ref{thm:main} and Theorem~\ref{thm:non-orientable}, following from~\cite{Ru} is:

\begin{corollary}
Any polynomial $P(z^2)=\displaystyle{\sum^n_{k=0} a_k z^{2k}}$ with non-negative integer coefficients can be realized as the coefficient of the top $a$-degree of the HOMFLY-PT polynomial of a knot. Any polynomial $P(z^2)$ or $zP(z^2)$ can be realized as the coefficient of the top a-degree term of the Kauffman polynomial of a knot.
\end{corollary}

\subsection{Smoothly non-isotopic fillings}

The majority of the exact Lagrangian fillings constructed to date are either smoothly isotopic relative to their boundary~\cite{CG,Gao2,Casals_weng,CN,CGGILS,Casals_zaslow} or fail to be homeomorphic~\cite{Cao_Gallup_Hayden_Sabloff_2014}. Exotic knotted surfaces in $\mathbb{B}^4$ with boundary a link $K\subset S^3$ are fillings of $K$ that are topologically isotopic relative to their boundary but not smoothly isotopic relative to their boundary. Obstructions for pairs of smooth fillings from being smoothly isotopic rel. boundary arise from gauge theory~\cite{Akbulut_1991,Hayden}, Heegard Floer homology~\cite{Juhasz_Zemke_2019,Juhasz_Miller_Zemke_2021} and more recently Khovanov homology~\cite{Hayden_Sundberg_2024,Hayden_Kim_Miller_Park_Sundberg_2025,Ian}. Li and Tange ~\cite{Li_Tange_2020} constructed examples of exact Lagrangian disks that are exotically knotted and relied on a gauge-theoretic obstruction of Lisca and Mati\'c~\cite{Lisca_Matic_1997}. We extend Li and Tange's approach to a broader class of fillings. We define two ribbon fillings $L_1, L_2$ to have $U$-compatible exteriors if they have exteriors $W_1, W_2$ with Kirby diagrams such that adding a particular set of $(-1)$-framed unknots to both $W_1$ and $W_2$ guarantees that $W_1$ admits a Weinstein structure and $W_2$ does not. See Definition~\ref{defn:compatible} for a precise definition of $U$-compatible exteriors.

\begin{theorem}\label{intro:prop:notsmooth} 
Let $L_1, L_2$ be two smooth immersed ribbon fillings of $\La$ with double points that have U-compatible exteriors. Then, $L_1$ and $L_2$ are smoothly non-isotopic relative to their boundary.
\end{theorem}

We apply Theorem~\ref{intro:prop:notsmooth} to various families of fillings of Legendrians, answer Question~\ref{ques:fillings} affirmatively, and show $\La_{P(z^2)}$ has smoothly non-isotopic fillings.

\begin{theorem}\label{thm:diff}
Let $n\leq 1$, and $k\geq 0$. There exist Legendrian knots $\La$ with $k$ embedded exact Lagrangian fillings $L$ with Euler characteristic $\chi(L)=n$ that are pairwise smoothly non-isotopic.
\end{theorem}

\begin{theorem}\label{thm:notsmooth}
Let $P(z^2)=\displaystyle{\sum^n_{k=0} a_k z^{2k}}$ be a polynomial with non-negative coefficients, and $\La_{P(z^2)}$ the Legendrian representative from Theorem~\ref{thm:main}. For each $0\leq k\leq n$ there exist $a_k$ immersed exact Lagrangian fillings $L_{j}$ of $\La_{P(z^2)}$ that are pairwise homeomorphic and not smoothly isotopic. Moreover, each $L_j$ induces a torus chart that covers one of the $a_k$ connected components $(\C^*)^{2k}$ of $\A ug(\La_{P(z^2)})$.
\end{theorem}

\begin{remark}
Some pairs of homeomorphic fillings in Theorem~\ref{thm:diff} and Theorem~\ref{thm:notsmooth} are exotic, but we expect that many are not since they may fail to be topologically isotopic. 
\end{remark}

The proof of Theorem~\ref{thm:notsmooth} does not rely on a general relationship between irreducible components of the augmentation variety and smooth isotopy classes of fillings. Rather, for the specific family $\La_{P(z^2)}$, we construct fillings associated to the various augmentation components and distinguish them using the obstruction of Theorem~\ref{intro:prop:notsmooth}. We pose the following conjecture motivated by the computational evidence provided by Theorem~\ref{thm:notsmooth}, and Example~\ref{ex:m821}:

\begin{conjecture}\label{conj:diff}
Let $\La$ be a Legendrian link whose (graded or ungraded) naive augmentation variety modulo dg-homotopy $\A ug(\La)$ has two or more (possibly intersecting) irreducible components. Suppose that $L_1$ and $L_2$ are two embedded exact Lagrangian fillings of $\La$ whose toric charts are contained in different irreducible components of $\A ug(\La)$. Then, $L_1$ and $L_2$ are not smoothly isotopic surfaces relative to their boundary.
\end{conjecture}

Rainbow closures of positive braids have irreducible augmentation varieties, and are conjectured to have restricted graded ruling polynomials~\cite[Conjecture 2.1]{SuXiYu}, prompting the following open question:

\begin{question}
Which polynomials can be realized as the graded ruling polynomial of a Legendrian $\La$ whose augmentation variety $\A ug(\La)$ is irreducible? 
\end{question}

Another classification problem of interest is that of symplectic fillings of transverse knots up to symplectic isotopy rel. boundary~\cite{Cao_Gallup_Hayden_Sabloff_2014,Baykur_Horn-Morris_2016, Oba_2020} which has applications to the study of closed symplectic surfaces and complex surfaces~\cite{Boileau_Orevkov_2001,Hayden}. Hayden~\cite[Theorem 3.1]{Hayden} constructed for every $g\geq 0$ an infinite family of transverse knots such that every knot bounds a pair of orientable exotic symplectic fillings of genus $g\geq 0.$ There is a standard way to smoothly approximate a (not necessarily exact) orientable Lagrangian filling of a Legendrian link $\La$ by a symplectic filling of the transverse pushoff of $\La$ (see~\cite[Lemma 4.1]{Cao_Gallup_Hayden_Sabloff_2014}). Thus, we obtain the following corollary of Theorem~\ref{thm:diff}:

\begin{corollary}\label{cor:symp}
Let $k,g\geq 0$. There exist transverse knots $K$ with $k$ embedded orientable symplectic fillings $L$ with genus $g$ that are pairwise smoothly non-isotopic.
\end{corollary}

Among the $k$ fillings, there are at least $\lfloor \frac{k}{2}\rfloor$ exotic pairs. However, we do not expect them all to be topologically isotopic rel. boundary. By applying the argument in~\cite[Proof of Theorem A]{Hayden} to the symplectic fillings in Corollary~\ref{cor:symp}, one can construct pairs of exotic smooth fillings of knots in $S^3$ that admit holomorphic embeddings into $\mathbb{B}^4.$

\subsection{Remarks on graded ruling polynomials and cluster structures on augmentation varieties}\label{sec:motivation}

We provide further context for the motivation behind Question~\ref{ques:aug_geography}. In some cases, graded normal rulings, and therefore the graded ruling polynomial, are known to encode ``cluster theoretic" information. For example, there are Legendrian knots $\La$ for which the graded ruling polynomial encodes information on $\mathbb{L}$-compressing cycles of embedded fillings of $\La$ and therefore information on ``potential" cluster structures on $\A ug(\La)$. In particular, the max-tb $7_4$ Legendrian has graded ruling polynomial $z^2$, an embedded filling of genus $1$, but no immersed filling with one double point. See~\cite[Section 3.1]{CLLMPT} for more examples of Legendrians with obstructed immersed fillings. Initial computations by the authors for small crossing Legendrians (see the Legendrian atlas~\cite{ChNg}) indicate that when two Legendrian links have the same graded ruling polynomial, their augmentation varieties are often homeomorphic.
See Section~\ref{sec:SZ} for examples of this phenomenon. Corollary~\ref{cor:diff_aug} demonstrates that this is in fact not generally true for all Legendrian links.

We now discuss how the graded ruling polynomial relates to the cluster structure on the augmentation variety for a restricted class of Legendrians. Let $w_0$ be the longest word in $S_n$ in the Bruhat order, and $\Delta \in Br^+_n$ the positive braid lift of the permutation $w_0$, so  $\Delta$ is a braid word for the half-twist. Suppose that $\beta\in Br_n^+$ has demazure product $\delta(\beta)=w_0$, that is that the positive braid contains a half twist. Let $\La(\beta\Delta)$ denote the $(-1)$ closure of $\beta\Delta$. The augmentation variety of $\La(\beta\Delta)$ is isomorphic to a braid variety $X(\beta)$ which was shown to be a cluster variety by~\cite{CGGILS} and~\cite{Galashin_Lam_Sherman-Bennett_2025}. Braid varieties admit weave decompositions~\cite{CGGS}. Henry-Rutherford~\cite{HR} showed that, more generally, the naive augmentation variety of a fixed front has a decomposition determined by graded normal rulings. In fact, the ruling decomposition on  $\A ug(\La(\beta\Delta))$ is equivalent to a weave decomposition of the isomorphic braid variety~\cite{JCHLLW}. Consequently, there exists an algorithm for computing the cluster variables of the maximal cluster torus of $X(\beta)$ through Morse complex sequences associated with graded normal rulings~\cite[Theorem 1.10]{JCHLLW}. The correspondence in~\cite[Theorem 1.2]{JCHLLW} matches graded normal rulings $\rho$ of $\La(\beta\Delta)$ and fillings $\mathfrak{w}_{\rho}$ described by Demazure weaves. The Legendrian $\La(\beta\Delta)$ has a unique graded normal ruling $\rho_0$ with a maximal number of switches corresponding to an embedded weave $\mathfrak{w}_{\rho_0}$~\cite[Proposition 2.11, Proposition 4.43]{JCHLLW}. Thus, the graded ruling polynomial $R_{\La(\beta\Delta)}(z)=\displaystyle\sum_{k=0}^na_kz^{2k}$ has $a_n=1$. Demazure weaves of $\La(\beta\Delta)$ are equivalent  up to mutation and weave equivalence moves~\cite[Theorem 4.12]{CGGS}. Therefore, the immersed Demazure weaves $\mathfrak{w}_{\rho}$ of $\La(\beta\Delta)$ can be obtained from $\mathfrak{w}_{\rho_0}$ of $\La(\beta\Delta)$ after a sequence of mutations and cycle deletions. Thus, the coefficient $a_{n-1}$ of $R_{\La(\beta\Delta)}(z)$ is equal to the number of immersed fillings with a single double point. The coefficient $a_{n-1}$ is also then equal to the number of $\mathbb{L}$-compressing cycles on the embedded filling $\mathfrak{w}_0$ which in turn is equal to the number of mutable vertices in the quiver $Q$. For $\La(\beta\Delta)=T(2,m)$, $m>2,$ one can verify that the coefficients $a_{n-k}$ for $k\geq 1$ encode the number of immersed fillings with $k$ double points, which in turn is equal to the number of $k$-tuples of vertices in $Q$ that do not share any edges. These examples suggest that the coefficients of the graded ruling polynomial may encode constraints on the cluster structure of the augmentation variety when it is irreducible which is the case for braid varieties. We do not formulate a precise conjecture here, but view this as an interesting direction for future work.

\subsection{Outline:} Section~\ref{sec:background} reviews the necessary background on Legendrian links, normal rulings, augmentations, and exact Lagrangian fillings. In Section~\ref{sec:SZ} we discuss SZ-pinch moves and their effect on the ruling polynomials and augmentation varieties. In Section~\ref{sec:construction} we construct the Legendrians $\La_{P(z^2)}$, verify $R_{\La_{P(z^2)}}(z)=P(z^2)$ and prove Theorem~\ref{thm:main} and Theorem~\ref{thm:non-orientable}. In Section~\ref{sec:aug_stack} we compute the augmentation varieties of $\La_{P(z^2)}$ and prove Theorem~\ref{thm:aug}. In Section~\ref{sec:exotic} we develop our tool to detect immersed ribbon fillings that are not smoothly isotopic rel. boundary and prove Theorems~\ref{thm:diff},~\ref{thm:notsmooth}, and~\ref{intro:prop:notsmooth}.

\subsection{Acknowledgments:} We thank Roger Casals, Daniela Cortes Rodriguez, James Hughes, Youlin Li, and Lenhard Ng for helpful conversations. YP is supported by NSFC Grant $0401300024$.

\section{Background}\label{sec:background}

\subsection{Legendrian links and ruling polynomials}
Let $\R^3_{std}$ denote $\R^3_{(x,y,z)}$ with the standard contact structure given by $\xi_{std}=ker(dz-ydx)$. A Legendrian link $\Lambda\subset \R^3_{std}$ is a smooth link such that $T_p\Lambda\subset \xi_{std}$ for all $p\in \La.$ The projection $\pi_F: \R^3_{(x,y,z)}\rightarrow \R^2_{(x,z)}$ is called the \emph{front projection}, and the image $\pi_F(\La)$ a \emph{front}. Call a subset of a front a \emph{subfront}. The projection $\pi_L: \R^3_{(x,y,z)}\rightarrow \R^2_{(x,y)}$ is called the \emph{Lagrangian projection}. The image of Legendrian link $\Lambda$ in $\pi_L$ is called a \emph{Lagrangian diagram} for $\Lambda$, and it is an immersed curve with zero oriented area. There is a standard way of passing from a front to a Lagrangian projection called the Lagrangian resolution (or Ng's resolution)~\cite[Section 2.1]{Ng_2003} where crossings map to crossings, and cusps are resolved as shown in Figure~\ref{fig:ng}.

\begin{figure}[!ht]
   
    \centering
    \includegraphics[width=2in]{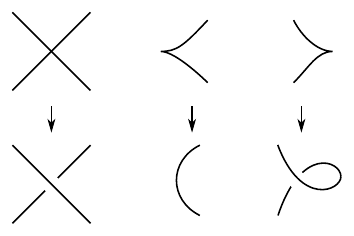}
    \caption{The Lagrangian resolution of a front.}
    \label{fig:ng}
\end{figure}

\begin{definition}\label{dfn:maslov}
    Let $\La\subset \R^3_{std}$ be a Legendrian link whose components have rotation number zero, and with fixed front $D$. A \textbf{Maslov potential} on a front $D$ is a locally constant function 
    \[
    \mu \colon D \setminus \{\text{cusps}\}\longrightarrow \Z
    \]
    that increases by one when we pass upwards through a cusp of $D$, and decreases by one when we pass downwards through a cusp of $D$.
    Call a front equipped with a Maslov potential a \textbf{graded front}.
\end{definition}

\begin{remark}
The Maslov potential is only well defined up to a shift of each connected component of $\La$. For links, we specify the Maslov potential on each link component by indicating the value on at least one strand of the link component. 
\end{remark}

\begin{definition}\label{defn:deg}
For a crossing $c$ in the graded front $D$ of a Legendrian, denote the upper strand (resp. lower strand) by $c_u$ (resp. $c_l$). 
The \textbf{degree} of $c$ is defined by $|c|=\mu(c_u)-\mu(c_l)$.
\end{definition}

\begin{definition}\label{defn:ruling}
 A \textbf{ungraded normal ruling} of the front $D$ of a Legendrian link $\Lambda \subset \R^3_{std}$ is a decomposition of $D$ into pairs of paths (called \emph{ruling paths}) where the two paths begin at the same left cusp and end at the same right cusp, and the pair of paths co-bound a topological disk (called a \emph{ruling disk}). The ruling paths are required to be smooth, except at cusps and certain crossings called \emph{switches}, and satisfy:
    \begin{enumerate}
        \item Any two ruling paths only intersect at crossings or cusps;
        \item Near a switch, pairs of ruling disks must look like one of the diagrams in Figure~\ref{fig:ruling_switches}.
    \end{enumerate}
     \begin{figure}[!ht]
        \centering
        \includegraphics[width=0.5\linewidth]{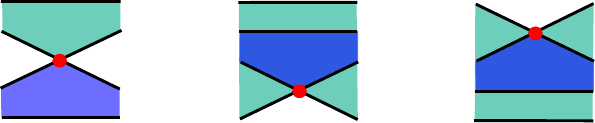}
        \caption{All possible configurations of a normal ruling near a switch.}
\label{fig:ruling_switches}
    \end{figure}
A \textbf{graded normal ruling} is an ungraded normal ruling on a graded front where switches can occur only at crossings with zero degree.
\end{definition}

\begin{definition}\label{defn:ruling_polynomial}
Let $\Lambda \subset \R^3_{std}$ be a Legendrian link with a fixed front $D$. The~\textbf{ungraded ruling polynomial}, $R^1_{\La}(z)$, is given by
$$R_{\Lambda}^1(z)=z^{|\pi_0(\Lambda)|}\sum_{\rho}z^{-\chi(\rho)}$$
where $\rho$ is an ungraded normal ruling of $\La$. The Euler characteristic of $\rho$ is $\chi(\rho):=d(\rho)-s(\rho)$, where $s(\rho)$ is the number of switches in $\rho$ and $d(\rho)$ is the number of ruling disks in $\rho.$ 

If $D$ is a graded front, the \textbf{graded ruling polynomial}, $R_{\Lambda}(z)$, is given by
$$R_{\Lambda}(z)=z^{|\pi_0(\Lambda)|}\sum_{\rho}z^{-\chi(\rho)}$$
where $\rho$ is a graded normal ruling of $\La$ and $\chi(\rho)$ is the Euler characteristic of $\rho$. 

\end{definition}

\begin{remark} The ungraded normal ruling is also referred to as the $1$-graded normal ruling first introduced by Chekanov and Pushkar~\cite{ChekanovPushkar}, so we use the notation $R_{\La}^1(z)$ for the ungraded ruling polynomial. From here on out, we write ruling instead of graded normal ruling unless we are working with an ungraded normal ruling. We also assume that all fronts are graded.
\end{remark}

\begin{example}
Let $H_{n}$ for $n\in \Z$ denote the Hopf link with the graded front shown in the rightmost diagram of Figure~\ref{fig:Hopf}. If $n\neq 0$, $H_n$ has a unique ruling. Only $H_0$ has ruling polynomial equal to $z^2(z^0+z^{-2})=z^2+1$.
\end{example}

 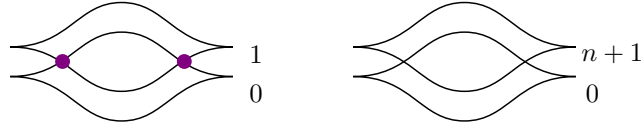
\begin{figure}[!ht]
        \centering
        \labellist
        \small
        \pinlabel $0$ at 130 20
        \pinlabel $1$ at 130 40
        \pinlabel $0$ at 300 20
        \pinlabel $n+1$ at 310 40
        \endlabellist
        \includegraphics[width=0.5\linewidth]{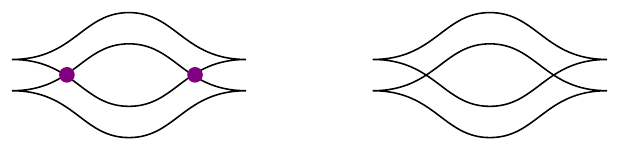}
        \caption{Rulings of the Hopf link $H_n$.}
        \label{fig:Hopf}
    \end{figure}

The following two properties of ruling polynomials follow directly from Definition~\ref{defn:ruling_polynomial}.
  
\begin{proposition}\label{prop:basic}
  \begin{enumerate}
    \item The graded and ungraded ruling polynomial of the max-tb unknot $\La_1$ is equal to $1$.
    \item The graded (ungraded) ruling polynomial of $\La \sqcup \La'$, is equal to the product of the graded (ungraded) ruling polynomials of $\La$ and $ \La'$.
\end{enumerate} 
\end{proposition}

\subsection{Augmentations of the Legendrian contact dg-algebra}\label{sec:aug}

The Legendrian contact dg-algebra $(\A(\La;R), \partial)$ is a unital graded algebra over a unital coefficient ring $R$ that is generated, freely commutatively, by the set of Reeb chords of $\La$. A Reeb chord of $\La$ is a trajectory of the Reeb vector field that begins and ends on $\La$. The grading on $\A(\La; R)$ is defined on the Reeb chord generators via a Conley-Zehnder index, and is defined on a word as the sum of the gradings of the letters in the word. The differential $\partial$ counts rigid $\mathcal{J}$-holomorphic disks in the symplectization with boundary on $\mathbb{R}\times \La$. There is a combinatorial method for computing $(\A(\La; \C[H_1(\La)]), \dd)$ from the Lagrangian resolution of a front $D.$ In this case, Reeb chords correspond to crossings  and the $\mathcal{J}-$holomorphic disks correspond to immersed polygons. See~\cite[Section 3]{EN_survey} or~\cite[Section 3]{CN} for more details.

 The Legendrian contact dg-algebra $(\A(\La;R), \dd)$ up to stable tame isomorphism is a powerful Legendrian invariant. 
 However, distinguishing two Legendrians directly from their dg-algebras is an undecidable problem~\cite{MR26}. Augmentations offer a more tractable invariant. 
 
\begin{definition}
 An \textbf{augmentation} of $(\A(\La; \C[H_1(\La)]),\dd)$ to $\C$ is a dg-algebra map to the ground dg-algebra $\e:(\A(\La; \C[H_1(\La)]),\dd)\rightarrow (\C, 0)$ where $\C$ is supported in degree $0$. That is, $\e$ is a unital algebra homomorphism such that
 \begin{enumerate}
     \item $\e(a)=0$ if $|a|\neq 0$, and
     \item $\e\circ \dd =0$.
 \end{enumerate}
Two augmentations $\e_1,\e_2$ of $\A(\La)$ are \textbf{dg-algebra homotopic} if there exists a degree $1$ $(\e_1,\e_2)$-derivation $H: \A(\La)\to \C$ such that $\e_1-\e_2= H\circ \dd$. An $(\e_1,\e_2)$-derivation $H: \A(\La)\to \C$ is a linear map such that $H(xy)= H(x)\e_2(y)+ (-1)^{|x|}\e_1(x)H(y)$.
\end{definition}

\begin{remark}
For $\La$ a link, there is another equivalence relation between two augmentations given by split dg-algebra homotopy, first introduced in~\cite[Definition 5.3]{CLLMPT}. A small error in the definition was corrected in~\cite[Definition 1.1]{Hamming_Gao}. We only use dg-algebra homotopy for consistency with the case where $\La$ is a knot; allowing us to compare augmentations between knots and links as in Section~\ref{sec:SZ}.
\end{remark}

\begin{definition}\label{defn:augmentation_stack}
Let $\La\subset \R^3_{std}$ be a Legendrian link with a fixed graded front $D$. Let the \textbf{augmentation variety} $\A ug(\La)$ denote the quotient of the naive augmentation variety of $D$ (over $\C$) modulo  dg-algebra homotopy.
\end{definition}

\begin{remark}\label{rem:aug_stack}
The augmentation variety is not generally an algebraic variety, but a quotient of an algebraic variety. From private correspondence with Casals and Ng, this quotient is also the 1-truncation of a derived augmentation stack defined as the moduli of pseudoperfect objects over the dg-category of perfect modules over the Legendrian contact dg-algebra.
\end{remark}

\begin{figure}[!ht]
    \centering
    \labellist
    \pinlabel $t$ at 255 130
  
    \pinlabel $a_1$ at 240 125
    \pinlabel $d_1$ at  50 82
    \pinlabel $c_1$ at 105 80
    \pinlabel $d_2$ at 150 80
    \pinlabel $c_2$ at 185 78
  
    \pinlabel $b_1$ at 0 100
    \pinlabel $b_2$ at 125 108
    \pinlabel $b_3$ at 217 105
    \pinlabel $a_2$ at 242 85
    \pinlabel $b_4$ at 0 70
    \pinlabel $b_5$ at 127 58
    \pinlabel $b_6$ at 215 55
    \pinlabel $a_3$ at 240 37
    
    \endlabellist
    \includegraphics[width=0.5\linewidth]{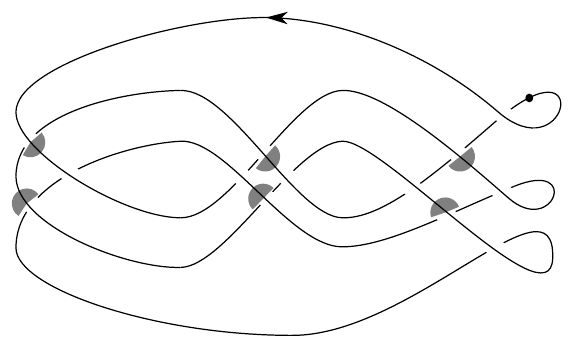}
    \caption{A Lagrangian projection of $\La_2$, the max-tb representative of $m(9_{46})$.}
    \label{fig:2lag}
\end{figure}
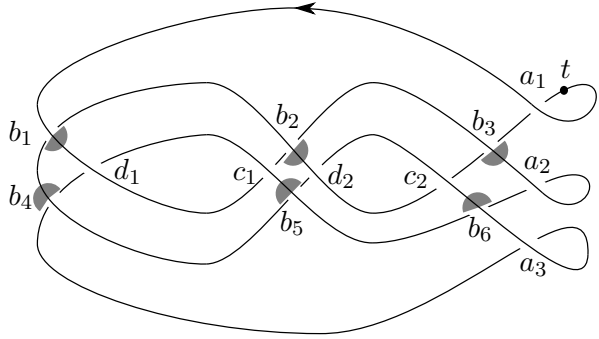

\begin{example}\label{ex:946}
The Legendrian contact dg-algebra of $\La_2$ is a $\C[t,t^{-1}]$ tensor algebra generated by the Reeb chords $$a_1,a_2,a_3, d_1,d_2, b_1, \ldots, b_6, c_1, c_2$$ as labeled in Figure~\ref{fig:2lag}. We have that
$|a_i|=|d_i|=1,~|b_i|=0,$ and $|c_i|=-1.$
To compute the augmentation variety it suffices to consider the truncated differential, $\partial_0$, with only words containing $0$ or $-1$ degree generators. For $\La_2$, the truncated differential is:

    \begin{align*}
        \partial_0 a_1 &= t^{-1}+b_1(1+b_3b_2)+b_3~~&\partial_0 d_1 &= b_1b_4\\
        \partial_0 a_2 &= 1+b_1b_6+b_3b_4~~&\partial_0 d_2 &=b_2b_1+b_4b_5+1\\ 
        \partial_0 a_3 &= 1+ (1+b_6b_5)b_4+b_6\\
        \partial_0b_2&=-b_4c_1~~&\partial_0b_5&=c_1b_1\\
        \partial_0 b_3 &= -b_1c_2~~&\partial_0 b_6 &= c_2b_4\\
    \end{align*}

Augmentations of $\La_2$ are determined by $\partial_0a_i=0$ and $\partial_0 d_i=0$, while dg-algebra homotopy equivalences are determined by $\partial_0 b_i$. By~\cite{Leverson_2016} any augmentation $\e$ of a knot must satisfy $\e(t)=-1.$ A direct computation shows that there are two augmentations to $\C$ up to dg-algebra homotopy equivalence which send the generators $b_i$ to the following values:
    
  \begin{equation}\label{eq:aug}
  \begin{array}{|c|c|c|c|c|c|c|c|}
  \hline
        \mbox{Augmentation} & b_1&  b_2& b_3& b_4& b_5& b_6 & t\\
        \hline
        \e_1 & 1&-1&0&0& 0&-1&-1\\
        \hline
        \e_2 &0 &0&1&-1&1&0&-1\\
        \hline
    \end{array}
    \end{equation}
  \noindent Thus, $\A ug(\La_2)$ is a disjoint union of two points. 

\end{example}

\subsection{Exact Lagrangian fillings}\label{sec:fillings}

See~\cite{EN_survey, survey_wiscon} for surveys on exact Lagrangian cobordisms and fillings.

\begin{definition}\label{defn:cobord} 
	Let $\Lambda_{+},\Lambda_{-}\sse(\R^{3}, \xi_{std})$ be Legendrian links.  
	An \textbf{exact Lagrangian cobordism} $L$ from $\Lambda_{-}$ to $\Lambda_{+}$ is an embedded Lagrangian surface in the symplectization $(\R_t\times \R^{3}, d(e^t(dz-y dx)))$
	   that  has  cylindrical ends  and is exact in the following sense:  
for some $N>0$, 
	\begin{enumerate}
		\item  $L \cap ([-N,N]\times \R^3)$ is compact,
		\item  $L \cap ((N,\infty)\times \R^3)=(N,\infty)\times \Lambda_{+}$ and $L\cap ((-\infty,-N)\times \R^3)=(-\infty,-N)\times \Lambda_{-}$, 
	\item there exists  a function $f: L \rightarrow \R$  and constants $\mathfrak c_\pm$ such that 
		$e^t\alpha|_{L} = df$, where $f|_{(-\infty, -N) \times \Lambda_{-}} = \mathfrak c_{-}$, and $f|_{(N, \infty) \times \Lambda_{+}} = \mathfrak c_{+}$.
	\end{enumerate}
	By definition, an \textbf{exact Lagrangian filling}  $L$ of a Legendrian link $\Lambda$ is an exact Lagrangian cobordism from the empty set $\emptyset$ to the Legendrian link $\Lambda$.
\end{definition}

There is a natural identification between exact Lagrangian fillings of $\La\subset \R^3_{std}$ in the symplectization of $\R^3_{std}$ and exact Lagrangian fillings in $(\mathbb{B}^4, \omega_{std})$ of $\La\subset (S^3, \xi_{std})$. Unless otherwise noted, fillings refers to exact Lagrangian fillings. All of the fillings we construct are decomposable, that is concatenations of elementary cobordisms including saddle cobordisms induced by pinch moves (see \cite[Section 6]{EHK} or \cite[Section 2]{CN}). Two Legendrians $\Lambda_{\pm}$ are related by a pinch move if the front of $\Lambda_{-}$ is given by taking two horizontal arcs in the front of $\Lambda_+$ and replacing them with two cusps. The location of a pinch move on a front can be indicated with a vertical line. If two Legendrians $\Lambda_-$ and $\Lambda_+$ are related by an orientable pinch move, then there exists an exact Lagrangian orientable saddle cobordism from $\Lambda_{-}$ to $\Lambda_+$. If they are instead related by a non-orientable pinch move, then there exists an exact Lagrangian non-orientable cobordism between $\Lambda_{-}$ and $\Lambda_{+}$ which is topologically an $\R\mathbb{P}^2$ with punctures union a collection of cylinders. See \cite[Section 4]{CN} for a description of the dg-algebra maps induced by decomposable Lagrangian cobordisms.

\subsection{SZ-pinch moves.}\label{sec:SZ}

In this section we describe how a set of moves called ``SZ-pinch moves" change the ruling polynomial $R_\La(z)$ and the augmentation variety $\A ug(\La)$. The main application is towards computing the ruling polynomial and augmentation variety without explicit long computations. See Figure~\ref{fig:szpinch} for an SZ-pinch move on a front. Unless otherwise stated, the Maslov potentials of the two components of $\La_+$ differ by $1$ so that the crossing $b$ is of degree $0$.

\begin{figure}[!ht]
    \centering
    \labellist
    \pinlabel $\La_+$ at 250 150
    \pinlabel $\La_-$ at 250 20
    \pinlabel $b$ at 100 148
    \pinlabel $b$ at 400 175
    \tiny{
    \pinlabel $m$ at -10 140
    \pinlabel $m+1$ at -20 170
    }
    \endlabellist
    
    \includegraphics[width=0.5\linewidth]{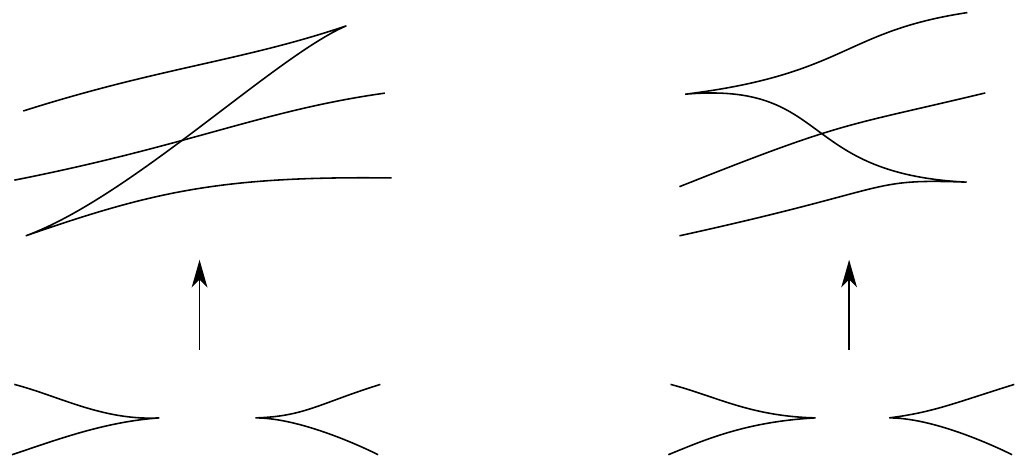}
    \caption{Two versions of the  SZ-pinch move.}
    \label{fig:szpinch}
\end{figure}
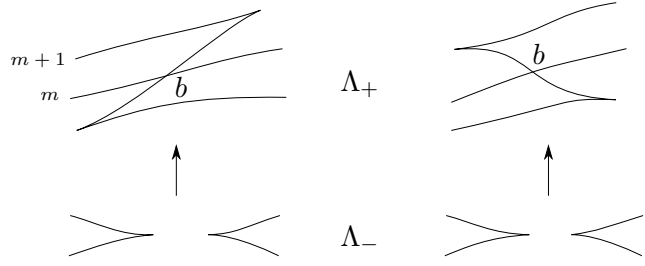

\begin{proposition}\label{prop:szpinch}
Let $\La_+,\La_-\subset \R^3_{std}$ be two Legendrian links whose fronts are related by an SZ-pinch move as shown in Figure~\ref{fig:szpinch}.

     \begin{enumerate}
         \item If $|\pi_0(\La_+)|=|\pi_0(\La_-)|+1$, then
         $$ R_{\La_+}(z)=z^2R_{\La_-}(z),~\text{and}~Aug(\La_+)~\text{is isomorphic to}~Aug(\La_-)\times (\C^*)^2$$
   
         \item If $|\pi_0(\La_+)|=|\pi_0(\La_-)|-1$, then
         $$R_{\La_+}(z)=R_{\La_-}(z),~\text{and}~Aug(\La_+)~\text{is isomorphic to}~ Aug(\La_-).$$
     \end{enumerate}
\end{proposition}
\begin{proof}

Since there is a decomposable Lagrangian cobordism $L$ from $\La_-$ to $\La_+$, every ruling $\rho_-$ of $\La_-$ induces a unique ruling $\rho_+$ of $\La_+$ with a switch at $b$. In fact, $\chi(\rho_+)=\chi(\rho_-)-1$. Then, the relations between the ruling polynomials of $\La_-$ and $\La_+$ follow directly from Definition~\ref{defn:ruling_polynomial}.

We prove the results for augmentation varieties. The dg-algebra map 
$$\Phi : ((\A(\La_+); \C[H_1(\La_+)]),\dd_+)\to ((\A(\La_-);\C[H_1(\La_-)]),\dd_-)$$ admits a factorization $\Phi= \phi_2\circ \phi_1$,
where 
$$\phi_1:((\A(\La_+); \C[H_1(\La_+)]),\dd_+)\to  ((\A(\La_-);\C[ H_1(\La_+), s^{\pm 1}]),\dd_-)$$
is the pinch map sending the Reeb chord $b$ to $s\in \C^{*}$ and fixing all other Reeb chords; and $$\phi_2:((\A(\La_-);\C[ H_1(\La_+), s^{\pm 1}]),\dd_-)\to ((\A(\La_-);\C[ H_1(\La_-)]),\dd_-)$$ is the dg-algebra map induced by the homomorphism on the coefficient ring 
$$\psi: \C[ H_1(\La_+), s^{\pm 1}])\to \C[ H_1(\La_-)].$$ 
By construction, the map $\phi_1$ is an isomorphism. If $|\pi_0(\La_+)|=|\pi_0(\La_-)|-1$, then $\psi$ is an isomorphism, and so is $\phi_2$. Therefore, $\Phi$ is an isomorphism between $\A(\La_+)$ and $\A(\La_-)$ and their augmentation varieties are isomorphic.  On the other hand, if $|\pi_0(\La_+)|=|\pi_0(\La_-)|+1$, then $\C[H_1(\La_+), s^{\pm 1}]$ is isomorphic to $\C[ H_1(\La_-)][u^{\pm 1}, v^{\pm 1}]$  for some invertible variables $u,v$. Hence the map $\psi$ obtained by forgetting the additional variables is surjective.
The isomorphism $\phi_1$ guarantees that $\A ug(\La_+)$ is isomorphic to $ \A ug(\La_-)\times (\C^*)^2$.
\end{proof}

\begin{remark}\label{rem:non-orientable_sz}
Proposition~\ref{prop:szpinch} can be modified for the case of non-orientable pinch moves as follows. Let $\mathbb{F}$ be a characteristic $2$ field. Let $\A ug_u(\La;\F)$ denote the ungraded augmentation variety (see~\cite[Definition 2.8]{CSHW}). Suppose that $\La_+, \La_-\subset \R^3_{std}$ are two Legendrian links whose fronts are related by an SZ-pinch move where the two components of $\La_+$ have Maslov potential that differs by any integer $n$. There are now three cases to consider. If $|\pi_0(\La_+)|=|\pi_0(\La_-)|$, then $R^1_{\La_+}(z)=zR^1_{\La_-}(z)$ and $\A ug_u(\La_+;\F)$ is isomorphic to $ \A ug_u(\La_-;\F)\times \F^*$. If $|\pi_0(\La_+)|=|\pi_0(\La_-)|+1$, then $R^1_{\La_+}(z)=z^2R^1_{\La_-}(z)$ and $\Aug_u(\La_+;\F)$ is isomorphic to $\A ug_u(\La_-;\F)\times (\F^*)^2$. Finally, if $|\pi_0(\La_+)|=|\pi_0(\La_-)|-1$, then $R^1_{\La_+}(z)=R^1_{\La_-}(z)$ and $\A ug_u(\La_+;\F)$ is isomorphic to $ Aug_u(\La_-;\F).$
\end{remark}

\begin{example}\label{ex:7_4}
    For the Legendrian link $\La^g$ in the checkerboard family for $g\geq 0$, defined in~\cite[Section 8.2]{CLLMPT}, there are $2g$ $SZ$-pinch moves to unknot. Thus, the ruling polynomial of $\La^g$ is $z^{2g}$ and its augmentation variety is the trivial cluster variety $(\C^*)^{2g}$. The Legendrian $\La^0$ is the max-tb unknot, and $\La^1$ is the max-tb $7_4$ knot.
\end{example}

\begin{figure}[!ht]
    \centering
    \includegraphics[width=0.7\linewidth]{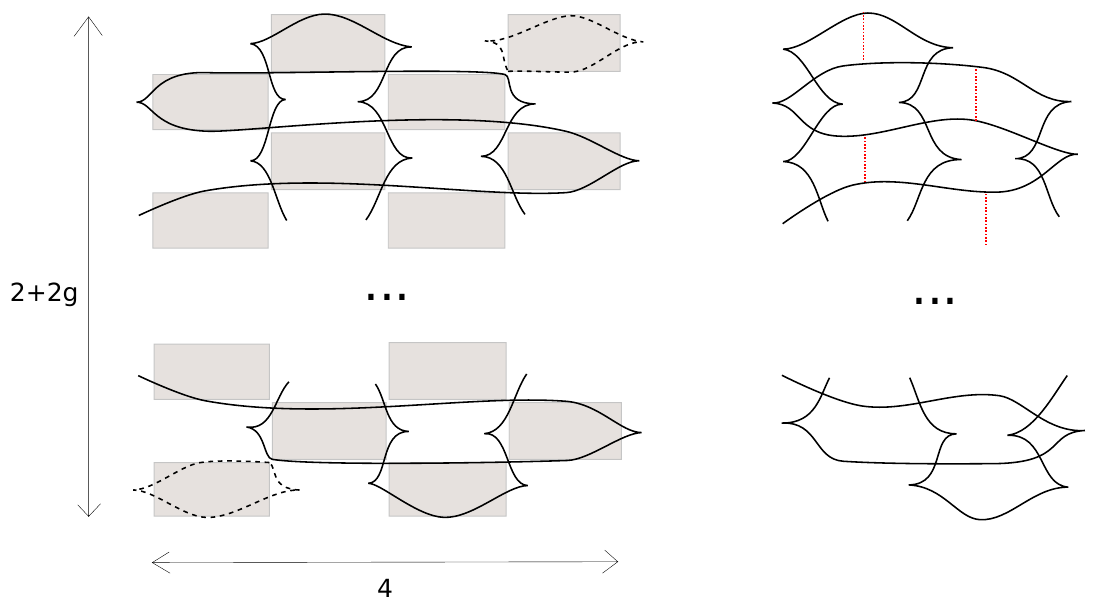}
    \caption{The front projection of $\La^g$ in the checkerboard family with an associated pinching sequence.}
    \label{fig:checkerboard}
\end{figure}

 \begin{figure}[!ht]
 \begin{minipage}{3in}
        \centering
        \labellist
        \small
        \pinlabel $2n$ at 120 25
        \endlabellist
        \includegraphics[width=0.8\linewidth]{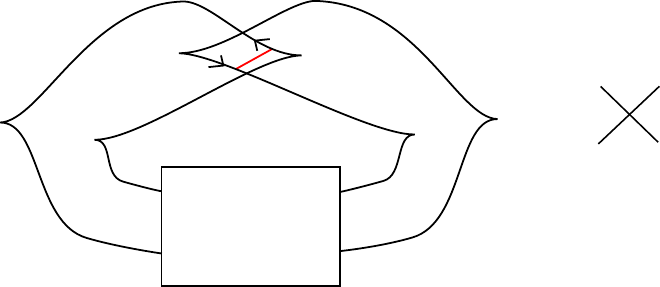}
        \caption{Positive fillable twist knots $K_{2n}$ for $n\geq 0$. The red band indicates the SZ-pinch move.}
        \label{fig:twist_pos}
\end{minipage}
\begin{minipage}{3in}
        \centering
        \labellist
        \small
        \pinlabel $|2+2n|$ at 112 25
        \pinlabel Z at 265 20
        \pinlabel S at 335 20
        \endlabellist
        \includegraphics[width=0.8\linewidth]{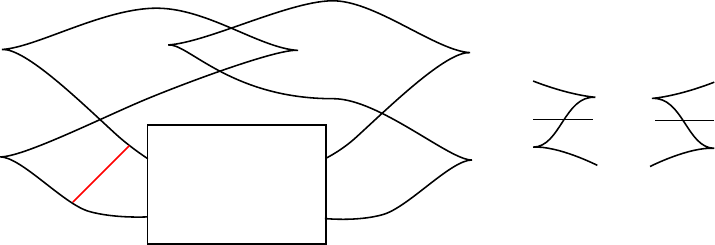}
        \caption{Negative fillable twist knots $K_{2n}$ for $n< 0$. The red band indicates the SZ-pinch move.}
        \label{fig:twist_neg}
         \end{minipage}
    \end{figure}

\begin{example}
Legendrian representatives of max-tb twist knots $K_{2n}$ have been classified~\cite[Theorem 1.1]{Etnyre_Ng_Vertesi_2013} (we follow their notation for positive and negative twist knots). There are two families of twist knots $K_{2n}$ that are exact Lagrangian fillable: $K_{2n}$ for $n\geq 0$ and for $n<0$. Positive fillable twist knots have a single max-tb representative shown in Figure~\ref{fig:twist_pos}, while negative fillable twist knots have $\lfloor \frac{n^2}{2}\rfloor$ max-tb representatives corresponding to different sequences of $|2n+2|$ $S,Z$ crossings as shown in Figure~\ref{fig:twist_neg}.

Gao-Rutherford~\cite{Gao_Rutherford_2021} studied a subfamily of negative fillable twist knots, $K_{2n}$ for $n<0$, and proved that their augmentation variety is isomorphic to the augmentation variety of the Hopf link $H_0$ over $\C$. SZ-pinch moves allow us to replicate this result and improve~\cite[Proposition 4.1]{Gao_Rutherford_2021} without having to explicitly compute the Legendrian contact dg-algebra. There exists a single SZ-pinch move from any max-tb representative $K_{2n}$ for $n<0$ to the Hopf link $H_k$, where $k\in \Z$ equals the difference between the number of $S$ and $Z$ crossings. Therefore, $\A ug(K_{2n})$ is isomorphic to $ Aug(H_k).$ The front of a max-tb positive fillable twist knot $K_{2n}$, $n\geq0$ is shown in Figure~\ref{fig:twist_pos}. A single orientable SZ-pinch move on $K_{2n}$ results in a negative $T(-2,2n)$ Legendrian torus link which only has ungraded augmentations. By Remark~\ref{rem:non-orientable_sz}, the ungraded augmentation variety $\A ug_u(K_{2n})$ is isomorphic to $\A ug_u(T(-2,2n))$ which is algebraically isomorphic to a $A_{2n-1}$ type cluster variety~\cite[Theorem 1.1]{CSHW}.
\end{example}

\section{Realization of ruling polynomials}\label{sec:construction}

 We prove Theorem~\ref{thm:main} in Subsection ~\ref{subsec:graded}  by constructing a Legendrian link $\La_{P(z^2)}$ with graded ruling polynomial $P(z^2)$ for each polynomial $P(z^2)$ with positive integer coefficients.
 We then characterize and realize ungraded ruling polynomials in Subsection~\ref{subsec:ungraded} where we prove Theorem~\ref{thm:non-orientable}.

\subsection{Graded ruling polynomials}\label{subsec:graded}
\begin{proof}[Proof of Theorem~\ref{thm:main}]

Choose $\La_1$ to be the max-tb unknot and $\La_{z^2}$ to be the checkerboard Legendrian $\La^1$ which is also the max-tb Legendrian $7_4$ from Example~\ref{ex:7_4}. By Proposition~\ref{prop:basic}, the ruling polynomial of a disjoint union of two Legendrian links is the product of their ruling polynomials so it suffices to consider polynomials with nonzero constant term.

In Proposition~\ref{lem:n}, we construct Legendrian links $\La_k$ with ruling polynomial $k\geq 1$. In Proposition~\ref{prop:2}, we construct Legendrian links $\La_{kz^{2n}+1}$ with ruling polynomial $kz^{2n}+1$ for $n,k\geq 1$.  To construct Legendrians for general polynomials with non-zero constant term, we define a new operation on fronts called XY-sum, that is only well defined for Legendrians with particular fronts. See Definition~\ref{defn:xytype} for Legendrians with $X$-type and $Y$-type fronts, $\La_X$, and $\La_Y$; see Definition~\ref{defn:sum} for the XY-sum between them, denoted $\La_X \ \natural\ \La_Y$. In particular, $\La_k$ and $\La_{kz^{2n}+1}$ for $k,n\geq 1$ both have $X$-type and $Y$-type fronts. By Proposition~\ref{prop:together}, if the ruling polynomials of $\La_X$ and $\La_Y$ are $p(z^2)+1$ and $q(z^2)+1$ respectively, then the ruling polynomial of $\La_X\ \natural \ \La_Y$ is $p(z^2)+q(z^2)+1$.

Let $P(z^2)=\displaystyle{\sum_{k=0}^N a_k z^{2n_k}}$ for $a_k\in\Z_{>0}$ and $0=n_0<n_1< n_2 < \cdots < n_N$ be an arbitrary polynomial. The Legendrian link $\La_{P(z^2)}$ is $$\La_{a_0}\ \natural \ \La_{a_1z^{2n_1}+1}\ \natural\ \La_{a_2z^{2n_2}+1}\  \natural \cdots \natural\  \La_{a_Nz^{2n_N}+1}.$$
\end{proof}

\begin{proposition}
\label{lem:n}
The Legendrian $\La_k$ for $k>1$ whose front is shown in Figure~\ref{fig:n} has ruling polynomial $R_{\La_k}(z)=k$.
\end{proposition}

\begin{proof}
For $k=2$, we choose $\La_2$ to be the max-tb $m(9_{46})$ whose front is shown in Figure~\ref{fig:9_46}.  It has two disk rulings, that we refer to as rulings with red and blue type switches. To construct $\La_3$, we stack the pattern of a subfront of $m(9_{46})$ twice with a small modification near the orange line, as shown in Figure~\ref{fig:3}. If the switches of a ruling of $\La_3$ in the upper layer are of red type, the corresponding switches in the lower layer can be of red or blue type. However, if the switches in the upper layer are of blue type, the corresponding switches in the lower layer must be of blue type. Therefore, $\La_3$ has three disk rulings. This construction can be generalized to obtain $\La_k$ which has $k-2$ additional layers as shown in Figure~\ref{fig:n}. Each iteration increases the number of disk rulings by one, resulting in $k$ disk rulings for $\La_k$. 

 \begin{figure}[!ht]
     \begin{minipage}{3in}
        \centering
        \includegraphics[width=3in]{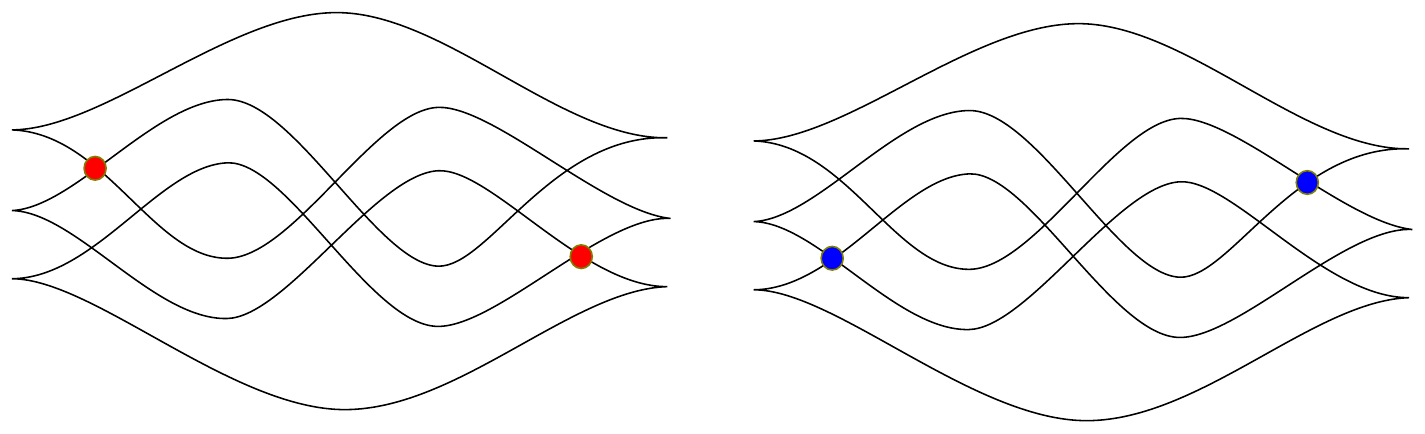}
        \caption{Two rulings of $m(9_{46})$. We say the ruling with red (blue) switches is of red (blue) type.}
        \label{fig:9_46}
    \end{minipage}
    \begin{minipage}{3in}
          \centering
        \labellist
        \pinlabel $k-2$ at 340 150
        \endlabellist
        \includegraphics[width=2in]{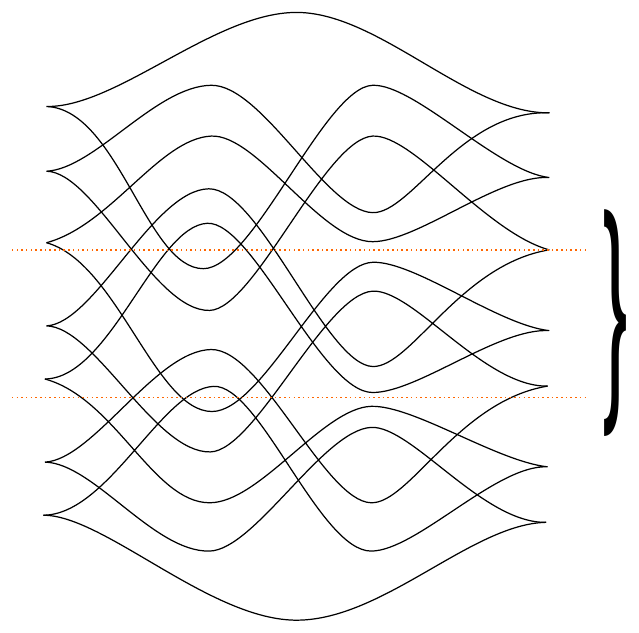}
        \caption{The front of $\La_k$ with $k-2$ orange lines.}
        \label{fig:n}
    \end{minipage}
    \end{figure}

\end{proof}
 \begin{figure}[!ht]
        \centering
        \includegraphics[width=0.8\linewidth]{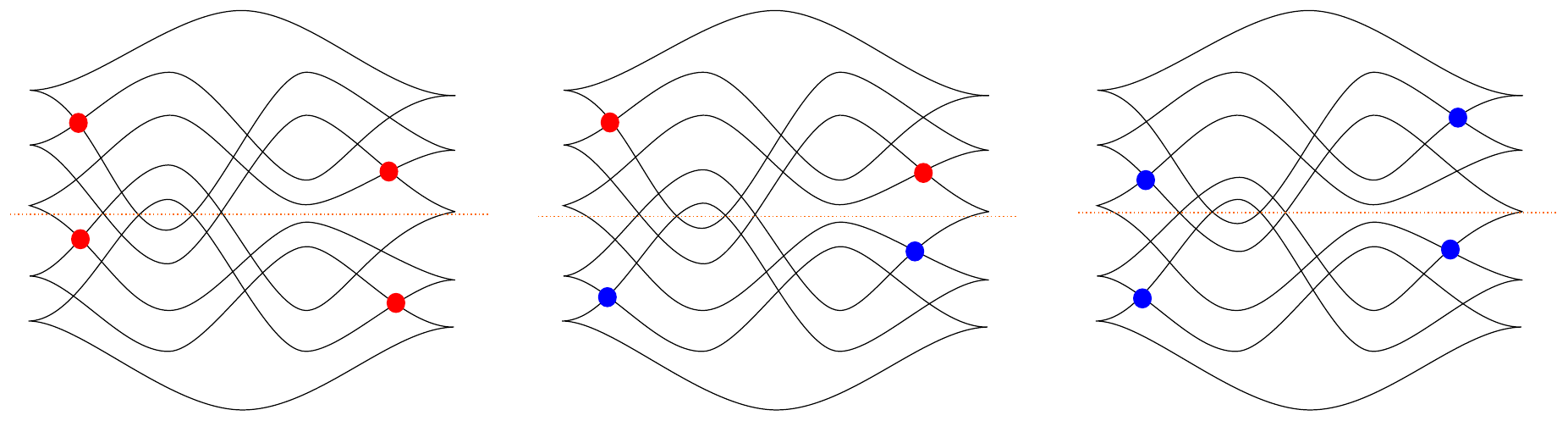}
        \caption{Three rulings of $\La_3$.}
        \label{fig:3}
    \end{figure}

We now define the first operation on fronts used to construct $\La_{P(z^2)}$. Note that the Legendrian link type after a crossed connect sum depends on the location of the operation.

\begin{definition}\label{defn:local_connect_sums}
Let $\La_1,\La_2\subset \R^3_{std}$ be two Legendrian links with  fronts $D_1, D_2$. After a Legendrian isotopy, assume that $D_1, D_2$ have subfronts as shown on the leftmost diagram in Figure~\ref{fig:local_moves}. A \textbf{crossed connect sum} of $\La_1$ and $\La_2$, $\La_1\sharp \La_2$, is the Legendrian whose front is given by modifying $D_1\sqcup D_2$ as shown in the rightmost diagram of Figure~\ref{fig:local_moves}. 
\end{definition}

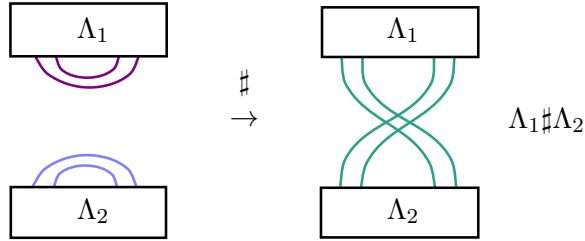
\begin{figure}[H]
 \begin{tikzpicture}[scale=1]
\node[inner sep=0] at (0,0){\includegraphics[width=0.4\textwidth]{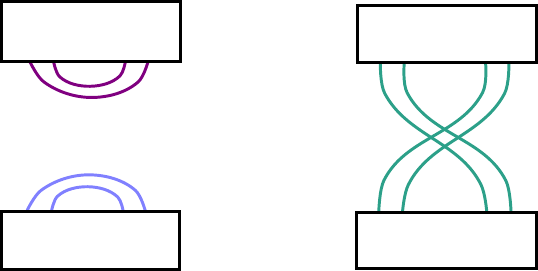}};
\node at (0,0.5) {{$\sharp$}};
\node at (0,0) {{$\rightarrow$}};
\node at (-2,1.2) {{$\La_1$}};
\node at (-2,-1.2) {{$\La_2$}};
\node at (2.1,1.2) {{$\La_1$}};
\node at (2.1,-1.2) {{$\La_2$}};
\node at (4, 0) {{$\La_1\sharp\La_2$}};
\end{tikzpicture}
\caption{A crossed connect sum of two Legendrians $\La_1, \La_2$.}
\label{fig:local_moves}

\end{figure}

\begin{proposition}\label{prop:2}
Let $n,k\geq 1$. There exists a Legendrian link $\La_{kz^{2n}+1}$ whose ruling polynomial is equal to $kz^{2n}+1.$    
\end{proposition}

\begin{proof}

We construct $\La_{kz^{2n}+1}$ by performing multiple crossed connect sums between $\La_{k+1}$, the Hopf link, and $n$ disjoint copies of $\La_2$ as shown in Figure~\ref{fig:kzn+1}. Note that $\La_{kz^{2n}+1}$ has $n+1$ components. It has two  types of rulings $\rho$:
\begin{enumerate}
    \item If the switches of $\rho$ on the top layer of $\La_{k+1}$ are of blue type, they are of blue type on all other layers of $\La_{k+1}$, and are of red type on all the other $\La_2$ as shown in part $(b)$ of Figure~\ref{fig:kzn+1}. In this case $\chi(\rho)=n+1$ and thus this ruling corresponds to the term $1$ in the ruling polynomial.
    \item If the switches of $\rho$ on the top layer of $\La_{k+1}$ are of red type, they can be of either of blue or red type on the second top-most layer of $\La_{k+1}$, and the switches on all other $\La_2$ are of blue type as shown in part $(a)$ of Figure~\ref{fig:kzn+1}. In this case $\chi(\rho)=1-n$ and there are $k$ many rulings of this type, each contributing a $z^{2n}$ to the ruling polynomial.
\end{enumerate}
Thus, the ruling polynomial of $\La_{kz^{2n}+1}$ is $kz^{2n}+1$.
\end{proof}

\begin{figure}[!ht]
    \centering
    \labellist
    \tiny
    \pinlabel $1$ at  250 150
    \pinlabel $1$ at 250 130
    \pinlabel $1$ at  0 155
    \pinlabel $1$ at -0 135
    \pinlabel $1$ at  405 150
    \pinlabel $1$ at 405 130
    \pinlabel $1$ at  695 150
    \pinlabel $1$ at 695 130
    \pinlabel $1$ at  440 155
    \pinlabel $1$ at 440 135
    \pinlabel $1$ at  845 150
    \pinlabel $1$ at 845 130
    \pinlabel $k-1$ at  125 85
    \pinlabel $k-1$ at  570 85
    \pinlabel $(a)$ at 130 0
    \pinlabel $(b)$ at  600 0
    \pinlabel $n$ at 265 40
    \pinlabel $n$ at 710 40
    \endlabellist
    \includegraphics[width=1\linewidth]{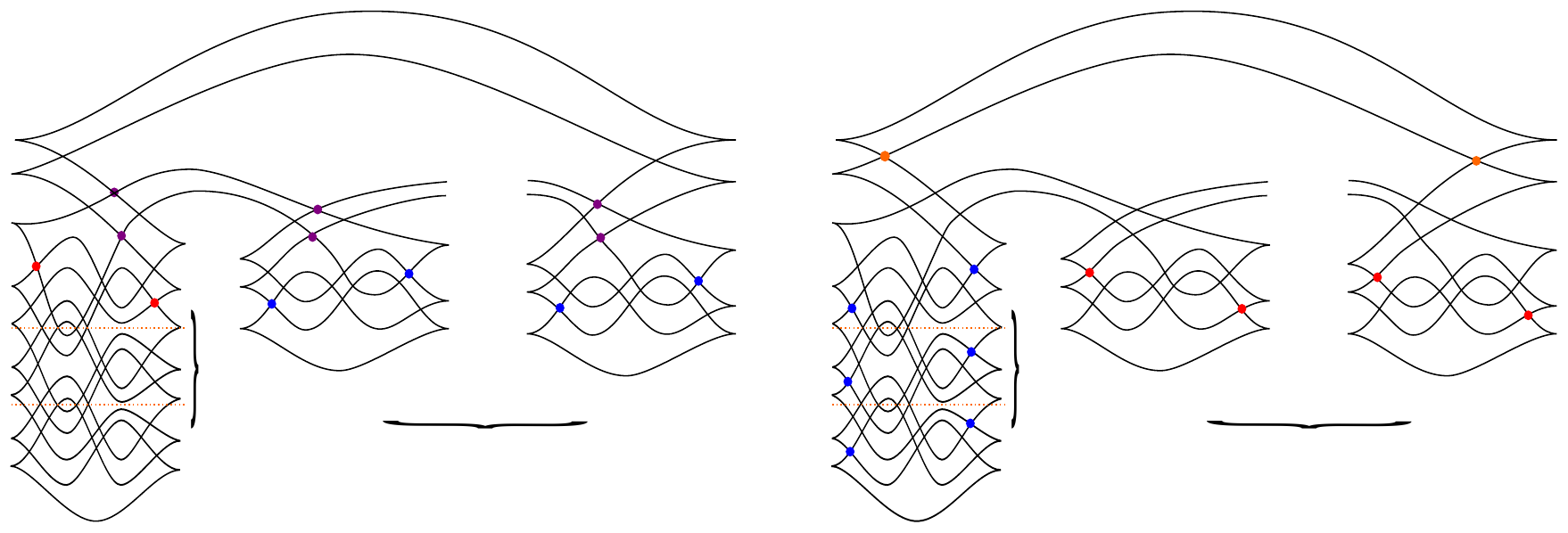}
    \vspace{0.05in}
    
    \caption{Two types of rulings of $\La_{kz^{2n}+1}.$}
    \label{fig:kzn+1}
\end{figure}
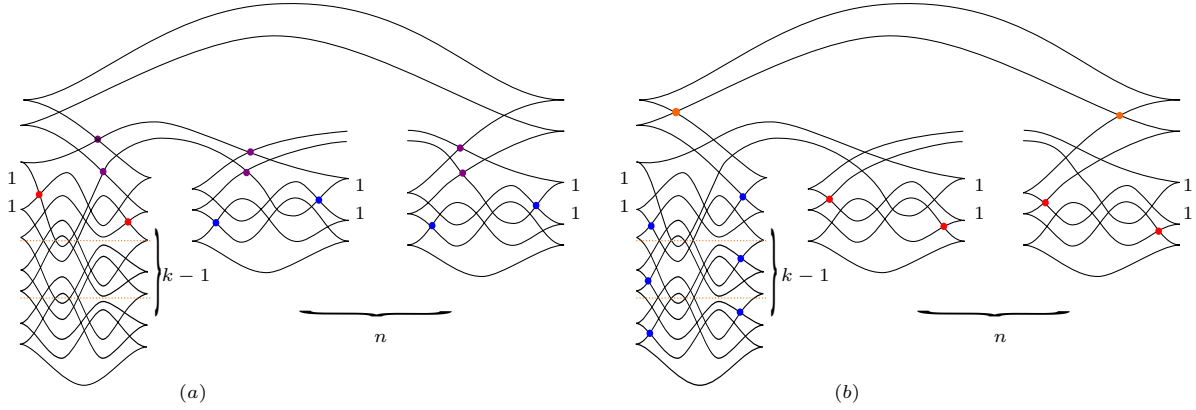

Our next operation only works for a restricted class of fronts.

\begin{definition}\label{defn:xytype}
Call a Legendrian front an X-type front, $\La_X$, if it has a subfront of the form shown in the top of Figure~\ref{fig:XYtype} with the following conditions:
\begin{enumerate}
    \item The subfront has the Maslov potential shown in the figure, rulings of red type, and possibly rulings of blue type;
    \item The ruling with red type switches contributes to $1$ in the ruling polynomial.
\end{enumerate}

We call a Legendrian front a $Y$-type front, $\La_Y$, if it has a subfront of the form shown in the bottom of Figure~\ref{fig:XYtype} with the following conditions:
\begin{enumerate}
    \item The top two strands belong to the same connected component, the subfront has the Maslov potential shown in the figure, and has two possible types of rulings, one of which we say is of orange type;
    \item The ruling labeled by orange type switches contributes to $1$ in the ruling polynomial.

\end{enumerate} 

\end{definition}

\begin{figure}[!ht]
    \centering
    \labellist
    \tiny
    \pinlabel $3$ at  15 185
    \pinlabel $2$ at 15 168
    \pinlabel $3$ at 125 185
    \pinlabel $2$ at 125 168
    \pinlabel $1$ at 15 145
    \pinlabel $1$ at 30 75
    \pinlabel $0$ at 40 62
    \pinlabel $X$ at  65 210
    \pinlabel $X$ at 270 210
    \pinlabel $Y$ at 70 20
    \pinlabel $Y$ at  270 20
    \pinlabel $(A)$ at -30 180
    \pinlabel $(B)$ at  -30 30
    \endlabellist
    \includegraphics[width=0.6\linewidth]{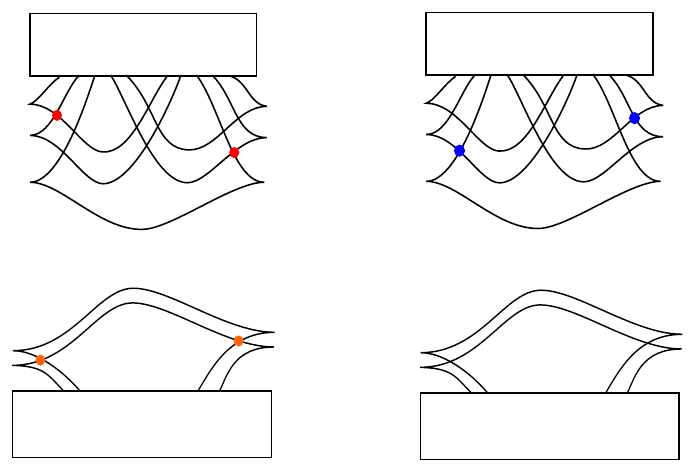}
    \vspace{0.1in}

    \caption{$(A)$ An X-type front; $(B)$ a Y-type front.} 
    \label{fig:XYtype}
\end{figure} 

\begin{figure}[!ht]
\includegraphics[width=0.2\linewidth]{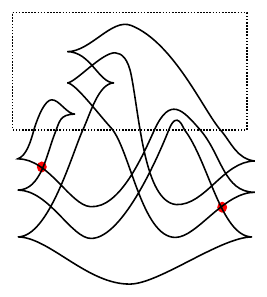}
\caption{An $X$-type front of the unknot $\La_1$.}
\label{fig:Xunknot}
\end{figure}

\begin{remark}
The Legendrians $\La_k,$ and $\La_{kz^{2n}+1}$ for $n,k\ge 1$ are X-type fronts.
The unknot $\La_1$ has an $X$-type front as shown in Figure~\ref{fig:Xunknot} with only one ruling of red type. For $k>1$, consider the bottom layer of $\La_k$ which can have rulings with red and blue type switches.  For $\La_{kz^{2n}+1}=\La_k\sharp\La_2\sharp\La_2$, $k,n\ge 1$, we choose the rightmost $\La_2$ factor.  The Legendrian fronts of  $\La_{kz^{2n}+1}$ are also $Y$-type fronts if you consider the upper portion of the front. 
\end{remark}

\begin{definition}\label{defn:stabilization}
For an $X$-type Legendrian front $\La_X$, the \textbf{stabilization} $S(\La_X)$ is the Legendrian obtained by modifying the $\La_X$ subfront to the subfront whose Lagrangian resolution is shown in Figure~\ref{fig:slag}.
\end{definition}

\begin{remark} The  stabilization of $\La_{k-1}$, $S(\La_{k-1})$, is $\La_k$ for $k\geq 3$.
\end{remark}

\begin{lemma}\label{lem:stabilization}
Let $\La_X$ be an $X$-type front Legendrian link and let $S(\La_X)$ be its stabilization. Let $R$ (resp. $B$) denote the rulings of $\La_X$ whose bottom layer switches are of red (resp. blue) type. 
Define
$r_R(z):=\displaystyle\sum_{\rho\in R}z^{-\chi(\rho)}$.
Then, 
$$R_{S(\La_X)}(z)=z^{|\pi_0(\La_X)|}r_R(z)+R_{\La_X}(z).$$
\end{lemma}

\begin{proof}
Define $r_B(z):=\displaystyle\sum_{\rho\in B}z^{-\chi(\rho)}$. 
By definition, $R_{\La_X}(z)=z^{|\pi_0(\La_X)|}(r_R(z)+r_B(z))$. It is straightforward to check that if $S(\La_X)$ has a ruling whose bottom layer is of red type, the layer above it must also be of red type. If $S(\La_X)$ has a ruling whose bottom layer is of blue type, the layer above it can be red or blue type. Thus, $R_{S(\La_X)}(z)=z^{|\pi_0(S(\La_X))|}(2r_R(z)+r_B(z))$. The claim follows from the fact that stabilizing does not change the number of components.
\end{proof}

\begin{definition}\label{defn:sum}
Given an X-type front $\La_X$ and a Y-type front $\La_Y$, the \textbf{XY-sum} $\La_X\ \natural \ \La_Y$ is the Legendrian front obtained by performing a crossed connect sum between $S(\La_X)$, $\La_Y$, and $\La_2$ as shown in Figure~\ref{fig:sum} .
\end{definition}

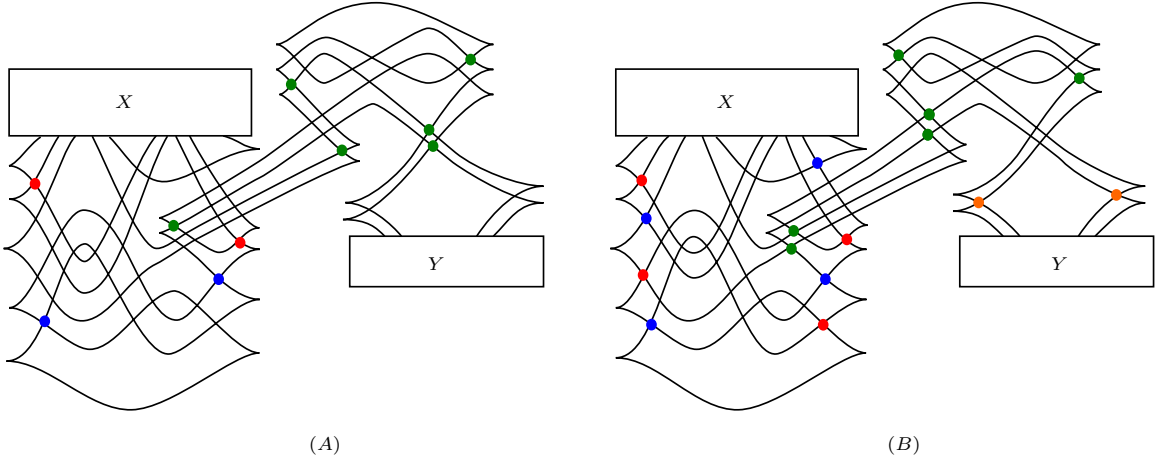
\begin{figure}[!ht]
    \centering
    \labellist
    \tiny
    \pinlabel $X$ at  60 145
    \pinlabel $Y$ at  200 72
    \pinlabel $X$ at  330 145
    \pinlabel $Y$ at  480 72
    \pinlabel $(A)$ at 150 -10
    \pinlabel $(B)$ at  410 -10
    \endlabellist
    \includegraphics[width=1\linewidth]{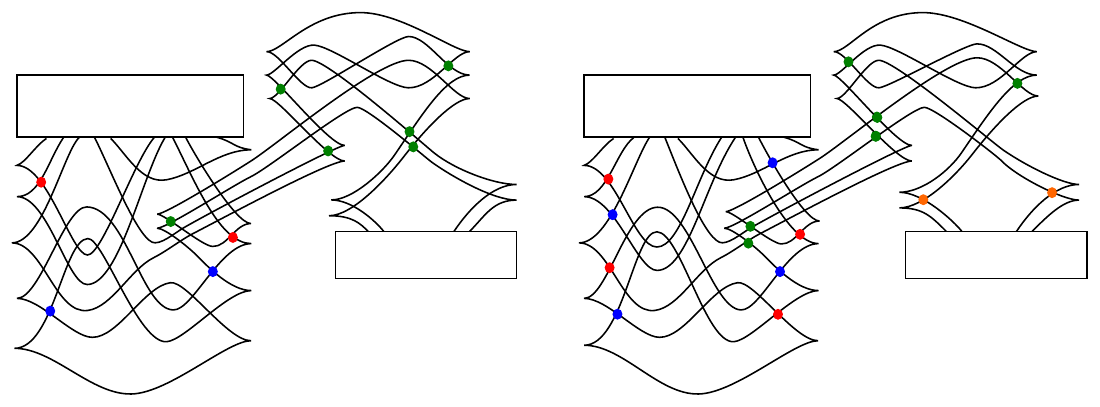}
    \vspace{0.05in}

    \caption{Three types of rulings contribute to $R_{\La_X \ \natural \  \La_Y}(z^2)=p(z^2)+q(z^2)+1$. $(A)$ the rulings contributing to $q(z^2)$ in the ruling polynomial. $(B)$ shows two types of rulings overlapped: the ruling given by red,  green, and orange switches contributes to $1$, while the rulings given by blue, green, and orange switches contribute to $p(z^2)$.}
    \label{fig:sum}
\end{figure} 
\begin{proposition}\label{prop:together}
Given an X-type Legendrian front $\La_X$ with ruling polynomial $p(z^2)+1$ and a Y-type Legendrian front $\La_Y$ with ruling polynomial $q(z^2)+1$ for polynomials $p(z^2),q(z^2)$, the ruling polynomial of $\La_X \ \natural\ \La_Y$ is $p(z^2)+q(z^2)+1$.
\end{proposition}

\begin{proof}
Observe that 
$$|\pi_0 (\La_X\ \natural\ \La_Y)|= |\pi_0( \La_X)| + |\pi_0 (\La_Y)|-1.$$

If a ruling $Y_\rho$ on $\La_Y$ has orange type switches and contributes to the constant term $1$ in $R_{\La_Y}(z)$, then the ruling $X_\rho$ on $\La_X$ can have either red or blue type switches as shown in $(B)$ of Figure~\ref{fig:sum}. The ruling with red type switches contributes to $1$ while the ruling with blue type switches contributes to $p(z^2)$ in $R_{\La_X}(z)$. The Euler characteristic of the glued rulings is $\chi (X_\rho) +\chi(Y_\rho)-1$, and they contribute $p(z^2)+1$ to the ruling polynomial of $\La_X\ \natural\ \La_Y$.

If a ruling $Y_{\rho}$ of $\La_Y$ does not have orange type switches, and contributes to $q(z^2)$ in $R_{\La_Y}(z)$, then $\La_X$ must have a ruling $X_{\rho}$ with switches of red type as shown in part $(A)$ of Figure~\ref{fig:sum}. Therefore, such rulings contribute to the constant term $1$ in $R_{\La_X}(z)$. We again have for the glued ruling that $\chi (X_\rho) +\chi(Y_\rho)-1$. Hence, these rulings contribute to $q(z^2)$ in the ruling polynomial of $\La_X\ \natural\ \La_Y$.

\end{proof}

\subsection{Ungraded ruling polynomials}\label{subsec:ungraded}

In this subsection, we characterize ungraded ruling polynomial and prove Theorem~\ref{thm:non-orientable}.

\begin{lemma}\label{prop:filling}
Let $\La\subset \R^3_{std}$ be a Legendrian link. If $L_1,L_2$ are exact immersed Lagrangian fillings of $\La$ with double points, then 
$$\chi(L_1)= \chi(L_2)~(mod~2).$$
\end{lemma}

\begin{proof}
Let $c_i$ be the number of double points in $L_i$ for $i=1,2$. Let $\tilde{L}_i$ be the embedded filling obtained from $L_i$ by Lagrangian surgery ~\cite{polterovich} on all of the double points. By~\cite{Chantraine_2010,Cao_Gallup_Hayden_Sabloff_2014}, we have that
\begin{align*}
    tb(\La)&=-\chi(\tilde{L}_i)
    -e(\tilde{L}_i)\\
    &=-\chi(L_i)+2c_i-e(\tilde{L}_i)
\end{align*}
where $e(\tilde{L}_i)$ is the normal Euler number. Since $e(\tilde{L}_i)\in 2\Z$ for $i=1,2$, we can conclude that
$\chi(L_1)=\chi(L_2)~(mod~ 2)$.
\end{proof}

\begin{proposition}\label{prop:parity}
Suppose $P(z)=\displaystyle\sum_{k=1}^Na_kz^{n_k}$ is a positive integer coefficient polynomial with $1\leq k_1,k_2\leq N$ such that $n_{k_1}\in2\Z,$ and $n_{k_2}\in 2\Z+1$. Then, there cannot exist a Legendrian link with ungraded ruling polynomial equal to $P(z).$
\end{proposition}

\begin{proof}

Suppose for a contradiction that $\rho_1$ and $\rho_2$ are two ungraded rulings of a Legendrian $\La$ such that $|\pi_0(\La)|-\chi(\rho_{1})=n_{k_1}\in 2\Z$ and $|\pi_0(\La)|-\chi(\rho_2)=n_{k_2}\in 2\Z+1$. Then, from the ruling $\rho_k$ for $i=1,2$ there exists immersed exact Lagrangian fillings $L_i$ with $c_i$ double points induced by $\rho_i$ such that $\chi(\rho_i)=\chi(L_i)$~\cite[Theorem 1.2]{Pan_Rutherford_2023}. Therefore, by Lemma~\ref{prop:filling},
\begin{align*}
    |\pi_0(\La)|-\chi(\rho_1)&=|\pi_0(\La)|-\chi(L_1)\\ &= |\pi_0(\La)|-\chi(L_2)~(\text{mod}~2) \\
    &= |\pi_0(\La)|-\chi(\rho_2).
\end{align*}
Then, $n_{k_1}=n_{k_2}~(mod~2)$ which is a contradiction. 
\end{proof}

We end by proving Theorem~\ref{thm:non-orientable} where we characterize which polynomials can be the ungraded ruling polynomial of a Legendrian. 
\begin{proof}[Proof of Theorem~\ref{thm:non-orientable}]
By Proposition~\ref{prop:parity}, the only polynomials that are not obstructed from being realized as ungraded ruling polynomials are of the form $P(z^2)$ and $zP(z^2).$ A direct computation shows that the graded and ungraded polynomials for the Legendrians $\La_{P(z^2)}$ are equal.
For $zP(z^2)$ we consider the max-tb $m(3_1)$ Legendrian $\La_{m(3_1)}$ with ungraded ruling polynomial equal to $z$. Then, by Proposition~\ref{prop:basic} the disjoint union of $\La_{m(3_1)}\sqcup \La_{P(z^2)}$ (where $\La_{m(3_1)}$ is on the right and $\La_{P(z^2)}$ is on the left) has ungraded ruling polynomial $zP(z^2).$ To obtain a link, one can take a connect sum $\La_{m(3_1)}\#\La_{P(z^2)}$ joining the rightmost cusp of $\La_{P(z^2)}$ to the leftmost cusp of $\La_{m(3_1)}$.

\end{proof}

\section{The augmentation varieties of $\La_{P(z^2)}$}\label{sec:aug_stack}

In this section, we compute (up to homeomorphism) the augmentation varieties of the Legendrians $\La_{P(z^2)}$. We first prove the case of $P(z^2)=k$ in Corollary~\ref{cor:aug_stack_n}, and the case of $P(z^2)=kz^{2n}+1$ in Proposition~\ref{prop:doublesharp}. After computing how the augmentation variety changes under the XY-sum in Proposition~\ref{prop:aug_natural}, we prove Theorem~\ref{thm:aug} for general polynomials $P(z^2)$ with non-negative coefficients.

We introduce some helpful notation. Let $\A ug^{\alpha}(\La)$ denote the subset of $\A ug(\La)$ that satisfies a condition $\alpha$. For example, consider an $X$-type front and use the labeling scheme from $\La_2$ on the bottom subfront of $\La_X$. As in the case of $\La_2$ there are two types of augmentations of $\La_X$: when $\e(b_1)=0$ and $\e(b_1)\neq 0$. Therefore, $\A ug(\La_X)$ can be decomposed into a disjoint union of two parts $\A ug(\La_X)= \A ug^{\e(b_1)=0}(\La_X)\sqcup \A ug^{\e(b_1)\neq0}(\La_X)$. In the case of the unknot $\La_1$, $\A ug^{\e(b_1)=0}(\La_1)$ is empty and $\A ug(\La_1)= \A ug^{\e(b_1)\neq 0}(\La_1)$ is a point.

We first show how the augmentation variety changes under stabilization (see Definition~\ref{defn:stabilization}).

\begin{lemma}\label{lem:stable}
Let $\La_X$ be an $X$-type front. The augmentation variety of $S(\La_X)$ is a disjoint union of three parts:
$$\A ug(S(\La_X))\cong \A ug^{\e(b_{1})=0}(\La_X)\sqcup\A ug^{\e(b_4)=0}(\La_X)\sqcup \A ug^{\e(b_4)=0}(\La_X).$$
\end{lemma}

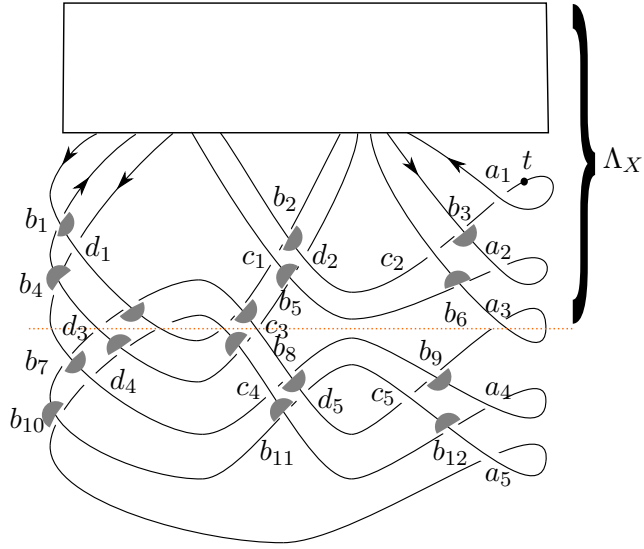
\begin{figure}[!ht]

    \centering
    \labellist
    \pinlabel $t$ at 265 205
    \pinlabel  $b_1$ at 10 174
    \pinlabel  $b_2$ at 139 185
    \pinlabel  $b_3$ at 230 180
    \pinlabel $a_1$ at 250 198
    \pinlabel $b_4$ at 7 140
    \pinlabel $b_5$ at 142 131
    \pinlabel $b_6$ at 228 125
    \pinlabel $a_2$ at 250 158
    \pinlabel $a_3$ at 250 128
    \pinlabel $d_1$ at  42 159
    \pinlabel $c_1$ at 123 153
    \pinlabel $d_2$ at 160 155
    \pinlabel $c_2$ at 195 152
    \pinlabel $d_3$ at 30 118
    \pinlabel $c_3$ at 135 117
    \pinlabel $b_7$ at 10 100
    \pinlabel $b_8$ at 139 107
    \pinlabel $b_9$ at 215 105
    \pinlabel $a_4$ at 250 85
    \pinlabel $b_{10}$ at 5 70
    \pinlabel $b_{11}$ at 135 53
    \pinlabel $b_{12}$ at 225 53
    \pinlabel $a_5$ at 250 40
    \pinlabel $d_4$ at 55 90
    \pinlabel $c_4$ at 120 85
    \pinlabel $d_5$ at 163 80
    \pinlabel $c_5$ at 190 82
    \pinlabel $\La_X$ at 315 205
    \endlabellist
    \includegraphics[width=0.5\linewidth]{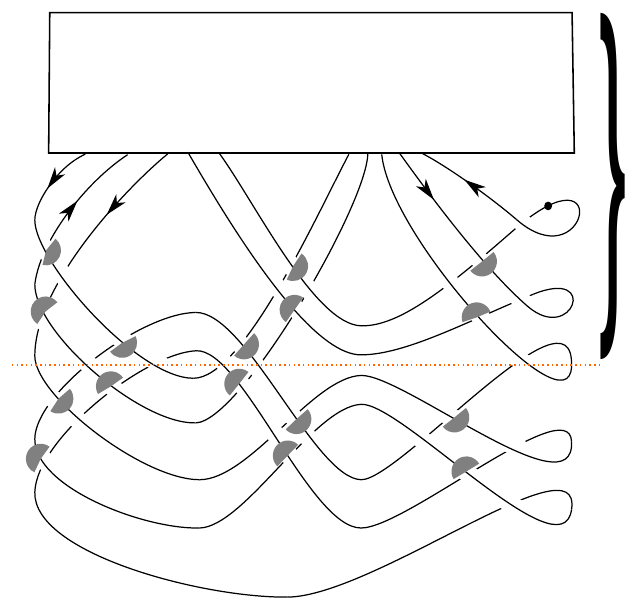}
    \caption{The Lagrangian projection of $S(\La_X)$. The shaded area corresponds negative sign.}
    \label{fig:slag}
\end{figure}

\begin{proof}
  
    See Figure~\ref{fig:slag} for the Lagrangian projection of $S(\La_X)$ and the labeled Reeb chords of degree $\pm1$ or $0$. Note that
    $$|a_i|=|d_i|=1, |b_j|=0, |c_i|=-1, \ 1\leq i \leq 5, 1\leq j\leq 12.$$
    For $a_1, a_2,b_2,b_3,b_5,$ and $b_6$ the truncated differential $\partial_0$ of $S(\La_X)$ coincides with $\dd_0^{\La_X}$ of $\La_X$. Thus, $\dd_0$ for $S(\La_X)$ is otherwise given by:

    \begin{align*}
       \dd_0 a_3 &= 1+(b_6 B +C)b_7(1+b_8b_9)+(\dd_0^{\La_X} a_3-1)b_9  
         & \dd_0 d_3 &= b_4b_7 \\
       \dd_0 a_4&=1+ b_9b_{10}+ b_7 b_{12}   
         & \dd_0 d_4 &= b_7 b_{10}\\
       \dd_0 a_5&=1+ b_{12}+b_{10}+ b_{12}b_{11}b_{10 } 
          & \dd_0 d_5&= 1+b_8b_7+ b_{10}b_{11}\\
      \dd_0 b_8 &= c_3b_4 -b_{10}c_4
      & \dd_0 b_9 &= -b_{7}c_5\\
      \dd_0 b_{11} &= c_4b_7
      & \dd_0 b_{12} &= c_5b_{10}
\end{align*}
where $A,B,C$ are words of Reeb chords of $\La_X$ such that the truncated differential $\dd_0^{\La_X}$ of $\A(\La_X)$ satisfies
$$  \dd_0^{\La_X} a_3 = 1+(A+b_6b_5)b_4+b_6 B +C.$$ 

After a direct computation, as in the case of $\La_2$, the bottom layer of $S(\La_X)$ has two types of augmentations: $\e(b_7)=0$ or $\e(b_{10})=0$. When $\e(b_{10})=0$, then $\e(b_4)=\e(b_9)=0$, and $\e(b_7)=1$, in which case $\dd_0$ agrees with $\dd_0^{\La_X}$ for the remaining generators. When $\e(b_7)=0$, $\e(b_9)=1$ and $\dd_0$ agrees with $\dd_0^{\La_X}$ with no additional restrictions. Thus, there are three augmentations of $S(\La_X)$ up to dg-algebra homotopy as follows:
  \begin{equation}\label{eqn:stab}
      \begin{array}{|c|c|c|c|c|c|c|c|c|c|c|c|c|c|}
      \hline
         & b_1&  b_2& b_3& b_4& b_5& b_6 &b_7& b_8& b_9& b_{10} &b_{11}&b_{12}&t\\
        \hline
        \e_{11} &1&-1&0&0&0&-1& 1&-1&0&0& 0&-1&-1\\
        \hline
        \e_{12} &1&-1&0&0&0&-1&0 &0&1&-1&1&0&-1\\
                \hline
        \e_{22} &0&0&1&-1&1& 0&0&0&1&-1&1&0&-1\\
        \hline
    \end{array}
    \end{equation}
    
    Observe that both $\e_{11}$ and $\e_{12}$ contribute to disjoint  
    copies of $\Aug^{\e(b_4)=0}(\La_X)$ in $\A ug(S(\La_X))$. 
Thus, $\A ug(S(\La_X))$ is homeomorphic to the disjoint union of 
$$\A ug^{\e(b_{1})=0}(\La_X)\sqcup\A ug^{\e(b_4)=0}(\La_X)\sqcup \A ug^{\e(b_4)=0}(\La_X).$$
\end{proof}

\begin{corollary}\label{cor:aug_stack_n}
    The augmentation variety of $\La_k$ is a disjoint union of $k$ points for $k\geq 1$.
\end{corollary}

\begin{proof}
The augmentation variety of $\La_1$, the max-tb unknot, is a single point. By Example~\ref{ex:946}, the augmentation variety of Legendrian $\La_2$ is a disjoint union of two points. By Lemma~\ref{lem:stable}, the augmentation variety of $\La_k=S(\La_{k-1})$ is a disjoint union of $k$ points.
\end{proof}

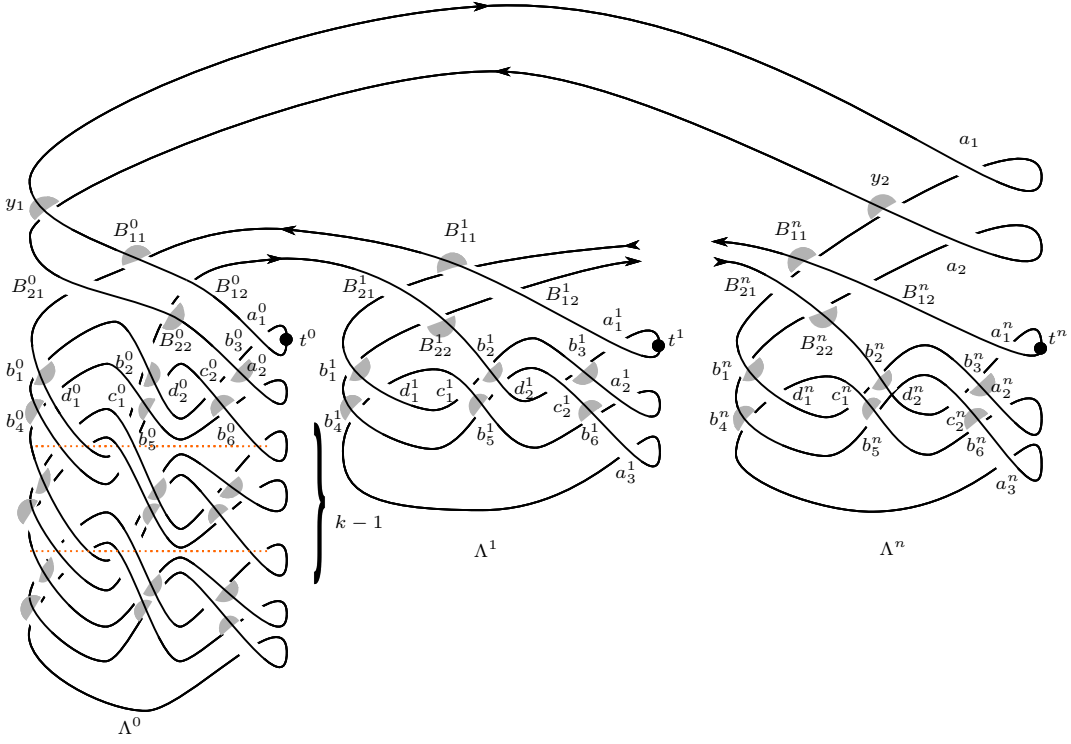
\begin{figure}[!ht]

\labellist
\tiny{
\pinlabel $k-1$ at 140 80
\pinlabel $a_1$ at 380 230
\pinlabel $a_2$ at 375 180
\pinlabel $y_1$ at 5 205
\pinlabel $y_2$ at 345 215
\pinlabel $B^0_{11}$ at 50 195
\pinlabel $B^0_{22}$ at 68 152
\pinlabel $B^0_{12}$ at 90 173
\pinlabel $B^0_{21}$ at 10 173
\pinlabel $B^1_{11}$ at 180 195
\pinlabel $B^1_{22}$ at 170 150
\pinlabel $B^1_{12}$ at 220 170
\pinlabel $B^1_{21}$ at 140 173

\pinlabel $B^n_{11}$ at 310 195
\pinlabel $B^n_{22}$ at 320 150
\pinlabel $B^n_{12}$ at 360 170
\pinlabel $B^n_{21}$ at 290 173

\pinlabel $a^0_1$ at 100 162
\pinlabel $a^0_2$ at 100 142
\pinlabel $t^0$ at 120 153
\pinlabel $b^0_1$ at 5 140
\pinlabel $b^0_2$ at 48 143
\pinlabel $b^0_3$ at 91 152
\pinlabel $b^0_4$ at 5 120
\pinlabel $b^0_5$ at 57 113
\pinlabel $b^0_6$ at 88 115
\pinlabel $d^0_1$ at 27 130
\pinlabel $c^0_1$ at 45 130
\pinlabel $d^0_2$ at 69 132
\pinlabel $c^0_2$ at 81 140

\pinlabel $a^1_1$ at 240 160
\pinlabel $a^1_2$ at 243 136
\pinlabel $a^1_3$ at 245 100
\pinlabel $t^1$ at 265 153
\pinlabel $b^1_1$ at 128 140
\pinlabel $b^1_2$ at 190 150
\pinlabel $b^1_3$ at 226 150
\pinlabel $b^1_4$ at 130 120
\pinlabel $b^1_5$ at 190 115
\pinlabel $b^1_6$ at 231 115
\pinlabel $d^1_1$ at 160 132
\pinlabel $c^1_1$ at 174 132
\pinlabel $d^1_2$ at 205 133
\pinlabel $c^1_2$ at 220 126

\pinlabel $a^n_1$ at 393 155
\pinlabel $a^n_2$ at 393 133
\pinlabel $a^n_3$ at 395 95
\pinlabel $t^n$ at 415 153
\pinlabel $b^n_1$ at 283 140
\pinlabel $b^n_2$ at 343 147
\pinlabel $b^n_3$ at 382 143
\pinlabel $b^n_4$ at 282 120
\pinlabel $b^n_5$ at 342 110
\pinlabel $b^n_6$ at 383 110
\pinlabel $d^n_1$ at 315 130
\pinlabel $c^n_1$ at 330 130
\pinlabel $d^n_2$ at 358 130
\pinlabel $c^n_2$ at 375 120
\pinlabel $\La^0$ at  50 0
\pinlabel $\La^1$ at 190 70
\pinlabel $\La^n$ at 350 70
}
\endlabellist
       \includegraphics[width=0.9\linewidth]{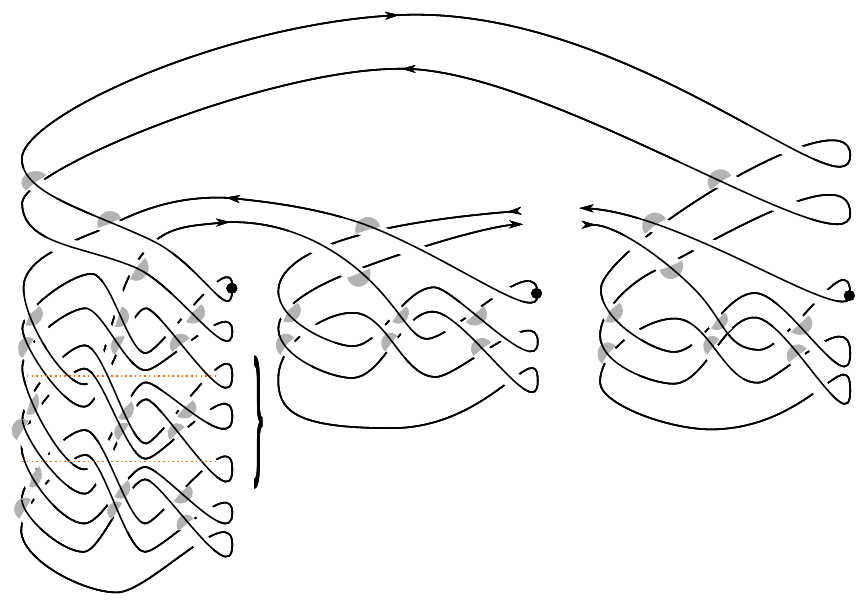}
                \vspace{0.1in}
                
        \caption{The Lagrangian projection of $\La_{kz^{2n}+1}$.}
        \label{fig:genelag}
    \end{figure}
    
\begin{proposition}\label{prop:kz2n+1}
The augmentation variety of $\La_{kz^{2n}+1}$ is a disjoint union of $k$ copies of $(\C^*)^{2n}$ and a point.
\end{proposition}
\begin{proof}
Label the Reeb chords of $\La_{kz^{2n}+1}$ as shown in Figure~\ref{fig:genelag}. Compared to $\sqcup_{i=0}^n\La^{i}$, where  $\La^0=\La_{k+1}$ and $\La^i=\La_2$ for $i=1, \cdots, n$,  $\La_{kz^{2n}+1}$ has $4(n+2)$ more Reeb chords with grading
$$|a_j|=1, \ |y_j|=|B^i_{jj}|=0,\ |B^i_{12}|=1, \ |B^i_{21}|=-1, \mbox{ for } j=1,2, \ i=0, \cdots, n.$$
The truncated differential $\dd_0$ of $\La_{kz^{2n}+1}$ that differs from the truncated differential of $\sqcup_{i=0}^n\La^{i}$
is listed below.

\begin{align*}
   \partial_0 a_1&=1+y_1 y_2+(-1)^{n+1}\Pi_{i=0}^{n}B_{11}^i& 
   \partial_0 a_2&=1+y_2 b_1^n +\Pi_{i=0}^{n}B_{22}^i\\
    \dd_0 a^0_1&=(t^0)^{-1}+y_1 b^0_3+B^0_{11}b^0_1 &
    \dd_0 a^1_1&=(t^1)^{-1}+(1+b^0_1b^0_2) (1+ b^1_2b^1_3)+B^1_{11}b^1_3\\
   \dd_0 a^0_2&= 1-B^0_{22}b^0_1b^0_6+b^0_3b^0_4 & 
     \end{align*}
       \begin{align*}
 \dd_0 a^i_1&=(t^i)^{-1}+b^{i-1}_1(1+ b^i_2b^i_3)+B^i_{11}b^i_3 &\mbox{ for }& i=2,\cdots ,n \\
\dd_0 a^i_2&= 1-b^i_3B^i_{22}b^i_4 + b^i_1b^i_6  &\mbox{ for }& i=1,\cdots ,n\\
\dd_0 d^i_2&= -B^i_{22}(1+b^i_4b^i_5)+b^i_2b^i_1 &\mbox{ for }& i=1,\cdots ,n
   \end{align*}

\begin{align*}
    \dd_0 y_2&=\sum_{k=0}^n(-1)^{n-k+1}(\Pi_{p=0}^{k-1} B^p_{22})B^k_{21} (\Pi_{q=k+1}^{n} B^q_{11})& \\
    \dd_0 b^0_3&=B^0_{22}b^0_1 c^0_2+B^0_{21}b^0_1&\\
    \dd_0 b^i_2&= B^i_{22}b^i_4c^i_1+B^i_{21} &\mbox{ for }& i=1,\cdots ,n\\
\dd_0 b^i_6&=- c^i_2B^i_{22}b^i_4 &\mbox{ for }& i=1,\cdots ,n
\end{align*}

\noindent Write $B^i_{\pm}$ as  $2\times 2$ matrices $$B^i_+=\left(\begin{array}{cc}
     B^i_{11}&  B^i_{12} \\
    B^i_{21}  &  B^i_{22}
\end{array}\right), B^i_-=\left(\begin{array}{cc}
    - B^i_{11}&  B^i_{12} \\
    B^i_{21}  &  B^i_{22}
\end{array}\right),$$
   
   \begin{align*}
   \dd_0 B_-^0&=\left(\begin{array}{cc}
     0&  y_1 \\
    0  &  0
\end{array}\right) B^0_+ + B^0_+\left(\begin{array}{cc}
     0&  1+b^0_1b^0_2 \\
    0  &  0
\end{array}\right)
 &
    \dd_0 B_-^1&=\left(\begin{array}{cc}
     0& 1+b^0_1b^0_2 \\
    0  &  0
\end{array}\right) B^1_+ + B^1_+\left(\begin{array}{cc}
     0&  b^1_1 \\
    0  &  0
\end{array}\right)\\
\dd_0 B_-^i&=\left(\begin{array}{cc}
     0& b^{i-1}_1 \\
    0  &  0
\end{array}\right) B^i_+ + B^i_+\left(\begin{array}{cc}
     0&  b^i_1 \\
    0  &  0
\end{array}\right)
&\mbox{ for }& i=2,\cdots ,n
     \end{align*}

It is straightforward to verify, working left to right, that if $\e(b_1^n)=0,$ and $\e(b_4^n)\neq 0$ then $\e(b_1^{i})=0$, and $\e(b_4^{i})\neq 0$ for $1\leq i\leq n$. Furthermore, $\e(y_1)=0$ and  $1+\e(b^0_1) \e (b^0_2)=0$, which implies that $\e(b^0_1)\neq 0$. In this case, the augmentation $\e$ depends on the values $\e(B^i_{11})=s_1^i$ and $\e(B^i_{22})=s_2^i$ where  $s_1^i,s_2^i\in\C^*$ for $i=1,\cdots, n$.
Similarly to the case of $\La_{k+1}$, there are $k$ such augmentations.

If $\e(b_1^n)\neq0,$ and $\e(b_4^n)= 0$ then $\e(b_1^{i})\neq 0,$ and $\e(b_4^{i})=0$ for $1\leq i\leq n-1$. Additionally $\e(y_1)\neq 0$, and $1+\e(b^0_1)\e(b^0_2)\neq 0$, which implies that  $\e(b_1^{0})= 0,$ and $\e(b_4^{0})\neq0$. All the $\e(B^i_{11})$  and $\e(B^i_{22})$ can be set to $0$ using dg-algebra homotopies for $0\leq i\leq n$. In fact, $\e$ is fixed on all generators in this case, and contributes to a point in $\A ug(\La_{kz^{2n}+1})$. In particular,
 $$\label{eq:esharp}
    \begin{array}{|c|c|c|c|c|c|c|c|c|c|c|c|c|c|c|}
     \hline
    &
    b_1^0&b_2^0&b_3^0&b_4^0&b_5^0&b_6^0& b_1^{i}&b_2^{i}&b_3^{i}&b_4^{i}&b_5^{i}&b_6^{i}&y_1&y_2
\\\hline
         \e_{b_1^n= 0}& \Pi_{i=1}^n s_2^i& -(\Pi_{i=1}^n s_2^i)^{-1} &0&0& 0&-1&0&0&-(s_2^i)^{-1}&-1&1&0&0&0
         \\\hline
         \e_{b_1^n\neq 0}& 0& 0 &1&-1&1&0&1&0&0&0&0&-1&1&-1\\  
\hline
\end{array}
$$

and
$$
    \begin{array}{|c|c|c|c|c|c|c|}
     \hline
   &B_{11}^0& B^0_{22}    & B_{11}^{i}& B^{i}_{22} &  t^0&t^i
\\\hline
         \e_{b_1^n= 0}& (-1)^n(\Pi_{i=1}^n s_1^i)^{-1} &-(\Pi_{i=1}^n s_2^i)^{-1}&s^i_1&s^i_2&  (-1)^{n+1}(\Pi_{i=1}^n(s_2^i)^{-1}s_1^i)  &s_2^i(s_1^i)^{-1}  
         \\\hline
         \e_{b_1^n\neq 0}&0&  0& 0& 0&-1&-1  
\\\hline
\end{array}
$$

In summary,
$$\A ug(\La_{kz^{2n}+1})=\A ug^{\e(b^n_1)=0}(\La_{kz^{2n}+1})\sqcup \A ug^{\e(b^n_1)\neq0}(\La_{kz^{2n}+1}) \cong (\sqcup_k(\C^*)^{2n})\sqcup pt.$$
\end{proof}

\begin{remark}\label{rem:XYsum}
To compute $\A ug(\La_X\ \natural \ \La_Y)$ we decompose $XY$-sums into two simpler operations: first add a layer of $m(9_{46})$ to $\La_X$ to obtain $S(\La_{X})$; and perform a double crossed connect sum of $\La_Y$ and $m(9_{46})$ as shown in Figure~\ref{fig:doublesharp}. Call such a Legendrian $S(\La_X) \sharp_{m(9_{46})} \La_Y$. If we perform additional Reidemeister I moves on $S(\La_X)$ and $\La_Y$ and then perform the double crossed connect sum we obtain the Legendrian shown in Figure~\ref{fig:doubleRI} which is in fact Legendrian isotopic to $\La_X  \ \natural \  \La_Y$  through a satellite isotopy such as the one shown in Figure~\ref{fig:sat}. This front is more convenient to work with in the proof of Theorem~\ref{thm:aug} and by abuse of notation we call both fronts $\La_X  \ \natural \  \La_Y$.
\end{remark}

\noindent Next, we compute the augmentation variety for a double crossed connect sum $\sharp_{m(9_{46})}$:

    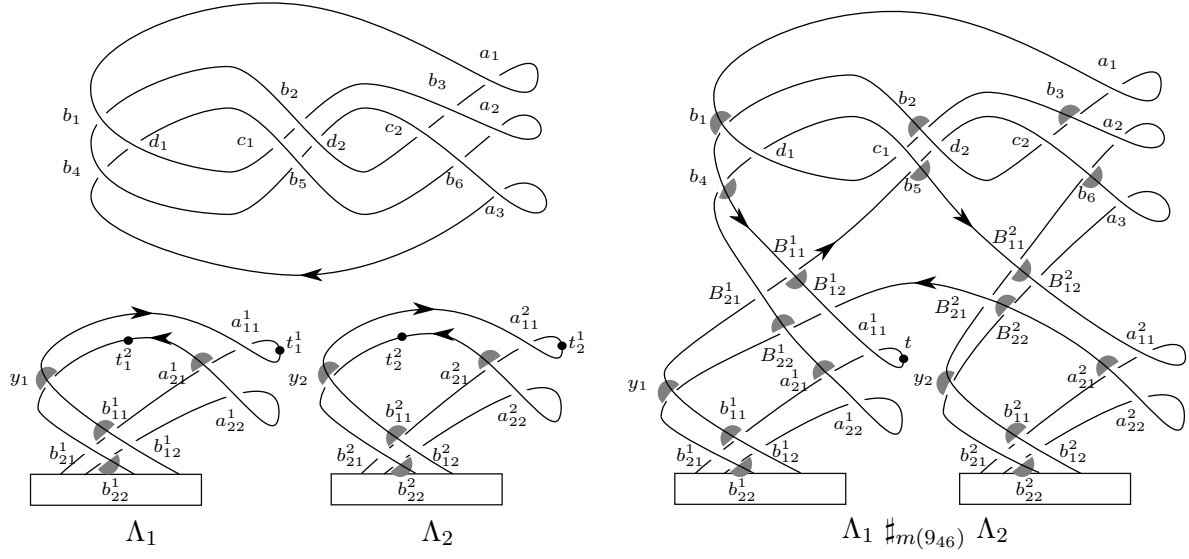
\begin{figure}[!ht]
        \centering
        \labellist
        \pinlabel $\La_1$ at 70 -10
        \pinlabel $\La_2$ at 240 -10
        \pinlabel $\La_1\ \sharp_{m(9_{46})}\ \La_2$ at 520 -10
        \tiny{
        \pinlabel $b_1$ at 30 230
        \pinlabel $b_2$ at 155 245
        \pinlabel $b_3$ at 240 250
        \pinlabel $b_4$ at 30 200
        \pinlabel $b_5$ at 160 193
        \pinlabel $b_6$ at 250 193
        \pinlabel $a_1$ at 270 265
        \pinlabel $a_2$ at 270 235
        \pinlabel $a_3$ at 273 175
        \pinlabel $y_1$ at 0 78
        \pinlabel $y_2$ at 160 78
        \pinlabel $a^1_{11}$ at 130 112
        \pinlabel $a^1_{22}$ at  121 55
        \pinlabel $a^1_{21}$ at 88 82
        \pinlabel $t_1^1$ at 158 100
        \pinlabel $t_1^2$ at 60 90
        \pinlabel $a^2_{11}$ at 290 114
        \pinlabel $a^2_{22}$ at 283 58
        \pinlabel $a^2_{21}$ at 250 82
        \pinlabel $t_2^1$ at 320 100
        \pinlabel $t_2^2$ at 215 90
        \pinlabel $d_1$ at 80 215
        \pinlabel $c_1$ at 130 215
        \pinlabel $d_2$ at 182 215
        \pinlabel $c_2$ at 215 222
        \pinlabel $b^1_{11}$ at 55 62
        \pinlabel $b^1_{22}$ at 55 15
        \pinlabel $b^1_{21}$ at 25 36
        \pinlabel $b^1_{12}$ at 85 40
        \pinlabel $b^2_{11}$ at 218 60
        \pinlabel $b^2_{22}$ at 225 15
        \pinlabel $b^2_{21}$ at 190 35
        \pinlabel $b^2_{12}$ at 245 35
        \pinlabel $y_1$ at 355 75
         \pinlabel $b_1$ at 390 230
        \pinlabel $b_2$ at 510 240
        \pinlabel $b_3$ at 595 245
        \pinlabel $b_4$ at 390 195
        \pinlabel $b_5$ at 513 190
        \pinlabel $b_6$ at 613 185
        \pinlabel $a_1$ at 625 260
        \pinlabel $a_2$ at 628 225
        \pinlabel $a_3$ at 630 170
        \pinlabel $y_2$ at 519 78
        \pinlabel $a^1_{11}$ at 490  110
        \pinlabel $a^1_{22}$ at  476 53
        \pinlabel $a^1_{21}$ at 445 78
        \pinlabel $t$ at 510 100
        \pinlabel $a^2_{11}$ at 644 105
        \pinlabel $a^2_{22}$ at 638 55
        \pinlabel $a^2_{21}$ at 610 80
        \pinlabel $B^1_{11}$ at 442 155
        \pinlabel $B^1_{22}$ at 435 92
        \pinlabel $B^1_{21}$ at 405 128
        \pinlabel $B^1_{12}$ at 465 133
        \pinlabel $B^2_{11}$ at 568 160
        \pinlabel $B^2_{22}$ at 570 105
        \pinlabel $B^2_{12}$ at 600 135
        \pinlabel $B^2_{21}$ at 535 120
        \pinlabel $d_1$ at 440 210
        \pinlabel $c_1$ at 495 210
        \pinlabel $d_2$ at 540 212
        \pinlabel $c_2$ at 575 215
        \pinlabel $b^1_{11}$ at 410 62
        \pinlabel $b^1_{22}$ at 415 15
        \pinlabel $b^1_{21}$ at 385 35
        \pinlabel $b^1_{12}$ at 440 37
        \pinlabel $b^2_{11}$ at 573 60
        \pinlabel $b^2_{22}$ at 580 15
        \pinlabel $b^2_{21}$ at 548 35
        \pinlabel $b^2_{12}$ at 605 37
        }
        \endlabellist
                \includegraphics[width=1\linewidth]{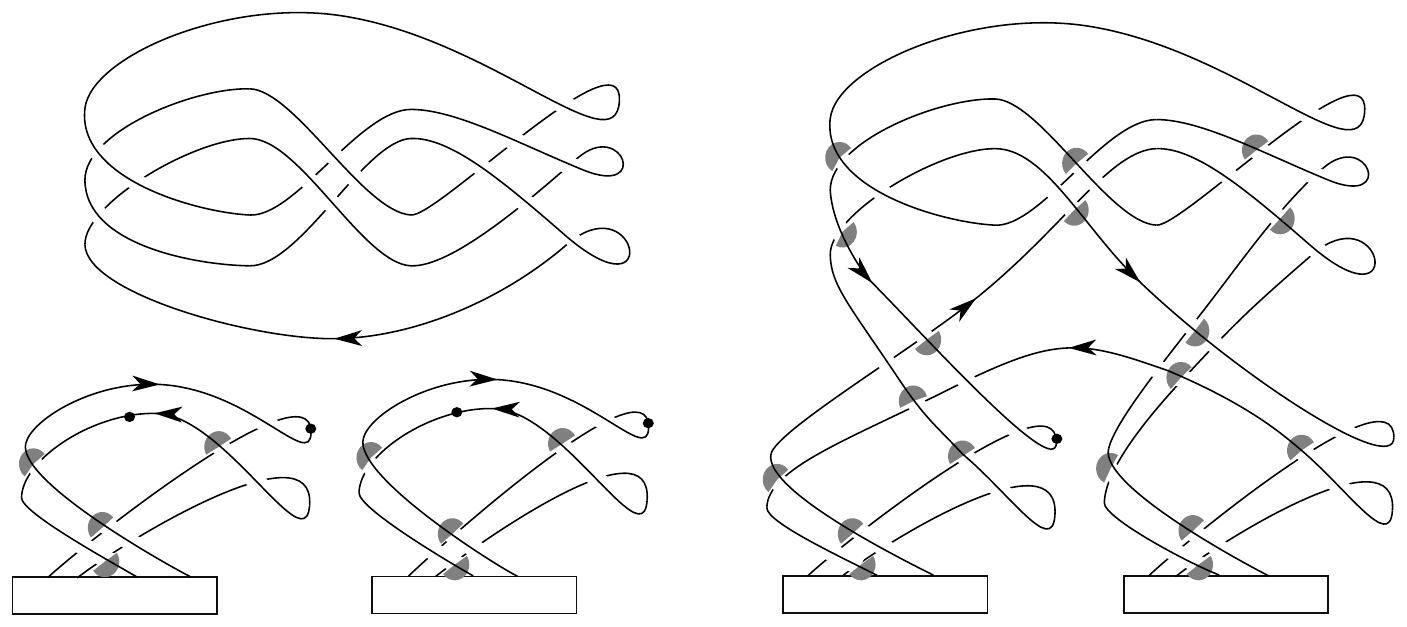}
                \vspace{0.1in}
                
        \caption{A local model of the double crossed connected sum $\sharp_{m(9_{46})}$ in the Lagrangian projection.}
        \label{fig:doublesharp}
    \end{figure}

\begin{proposition}\label{prop:doublesharp}
Let $\La_1,\La_2\subset \R^3_{std}$ for $i=1,2$ have subfronts as shown in Figure~\ref{fig:doublesharp}.
We further assume that the two strands of $\La_i$ appearing in the local model shown in Figure ~\ref{fig:doublesharp} belong to the same connected component of $\La_i$, and their Maslov potential differs by $1$ so that the Reeb chords $y_i$ are of degree $0$. 
Then,
    $$\A ug(\La_1 \ \sharp_{m(9_{46})}\  \La_2)= \A ug^{\e(y_1)\neq 0}(\La_1)\times \A ug^{\e(y_2)=0}(\La_2)  \sqcup \A ug^{\e(y_1)=0}(\La_1)\times \A ug^{\e(y_2)\neq 0}(\La_2).
    $$
\end{proposition}
    \begin{proof}
    Since $|\pi_0(\La_1\  \sharp_{m(9_{46})} \ \La_2)|= |\pi_0(\La_1)|+|\pi_0(\La_2)|-1$, we only decorate the diagram in Figure~\ref{fig:doublesharp} with one base point $t$. Observe that $\La_1\ \sharp_{m(9_{46})} \ \La_2$ has additional Reeb chords in comparison to $\La_1\sqcup \La_2\sqcup m(9_{46})$ with degrees
    $|B^j_{11}|=|B^j_{22}|=|B^j_{12}|-1=|B^j_{21}|+1=0, j=1,2.$
        The truncated differential of $\La_1\ \sharp_{m(9_{46})} \ \La_2$ is given by:
        \begin{align*}
             \dd_0 a_1&= 1-b_1-b_3-b_1b_2b_3 
             &\dd_0 b_2&=-b_4c_1\\
             \dd_0 a_2&=1+b_3b_4B^2_{11}+b_1B^1_{11}b_6
             &    \dd_0 b_5&= c_1b_1B^1_{11}+B^1_{21}\\
             \dd_0 a_3&=1+b_6y_2+y_1B^2_{22}
             &  \dd_0 b_3&=-b_1B^1_{11}c_2\\
             \dd_0 d_1&=b_1b_4
             &   \dd_0 b_6&=c_2b_4B^2_{11}-y_1B^2_{21}\\
             \dd_0 d_2&=(1+b_2b_1)B^1_{11}+b_4b_5
             &&
             \end{align*}  
            
             \begin{align*}
             \dd_0 B_-^1&= \left(\begin{array}{cc}
                  0&-b_4  \\
                  0&0 
             \end{array}\right)B_+^1+ \left(\begin{array}{cc}
                  0&-y_1  \\
                  0&0 
             \end{array}\right) B_+^1\\
            \dd_0 B_-^2&= \left(\begin{array}{cc}
                  0&b_5y_1-B_{22}^1  \\
                  0&0 
             \end{array}\right)B_+^1+ \left(\begin{array}{cc}
                  0&-y_2  \\
                  0&0 
             \end{array}\right) B_+^1\\
        \dd_0 b_-^i&= \begin{pmatrix} 0 & y_i\\ 0&0\end{pmatrix}b_+^i+ b_+^i\begin{pmatrix} 0 & C_i\\ 0&0 \end{pmatrix}
        \end{align*}

Let $\bar{x}_0^i$ denote the $2\times 2$ matrices:
            $\begin{pmatrix} -x_{11}^i & 0\\ x_{21}^i & x_{22}^i \end{pmatrix},$
        \begin{align*}
            \dd_0 \left(\begin{array}{cc}
                  a^1_{11}&0 \\
                  a^1_{21}&a^1_{22} 
             \end{array}\right) &= \left(\begin{array}{cc}
                  t+b_2a^1_{21}&0  \\
                  0& 1+a^1_{21}C_1 
             \end{array}\right) + \bar{B}_0^1\bar{b}_0^1\\
            \dd_0 \left(\begin{array}{cc}
                  a^2_{11}&0 \\
                  a^2_{21}&a^2_{22} 
             \end{array}\right)&= 
              \left(\begin{array}{cc}
                  1-(b_5y_1-B^1_{22})a^2_{21}&0  \\
                  0& 1+a^2_{21}C_2 
             \end{array}\right)+\bar{B}^2_0\bar{b}_0^2.
        \end{align*}

The $C_i$ are words of Reeb chords in $\La_i$ not shown such that the truncated differential $\dd_0^{\La_i}$ of $\La_i$ is
              \begin{align*}
             \dd_0^{\La_i} \left(\begin{array}{cc}
                  a^i_{11}&0 \\
                  a^i_{21}&a^i_{22} 
             \end{array}\right)&= 
              \left(\begin{array}{cc}
                  t^1_i+y_i(t_i^2)^{-1}a^i_{21}&0  \\
                  0& 1+a^i_{21}C_i 
             \end{array}\right)+
             \left(\begin{array}{cc}
                1&0  \\
                  0& t^2_i
             \end{array}\right)\bar{b}_0^i
             \end{align*}

There are two types of augmentations $\e(b_1)=0$ and $\e(b_1)\neq 0$:

$$\begin{array}{|c|c|c|c|c|c|c|c|c|c|c|c|c|c|c|c|}
\hline
         &b_1&b_2&b_3&b_4&b_5&b_6 &y_1&y_2&B^1_{11}& B^1_{22} & a^1_{21}&B^2_{11}& B^2_{22} & a^2_{21}
         \\  \hline 
\e_{b_1=0}&0&0&1&-s_1^{-1}&0&0&-s_2^{-1}&0&0&0&s_1t&s_1&s_2&0\\
\hline
         \e_{b_1\neq 0}&1&-1&0&0&0&-s_1^{-1}& 0&s_1&s_1& s_2 & 0 &0&0&-s_2^{-1}\\
         \hline
\end{array}$$

Let $\e(b_1)=0,$ then $\e(b_3)=1$, and $\e(B^2_{11})\e(b_4)=-1$. Since $\e(B^2_{11}):=s_1\neq 0$, then $\e(b_4)=-s_1^{-1}$. It follows that a dg-algebra homotopy takes $\e(b_2), \e(b_5),$ and $\e(b_6)$ to $0$. 
From the differential of $B^2_{12}$, $\e(B^2_{22}):s_2\neq 0$, and $\e(y_1)=-s_2^{-1}$. A dg-algebra homotopy sets $\e(B^1_{11})$, $\e(B^2_{22})$, and $\e(a^2_{21})$ equal to $0$. Then, $\e(y_2)=0$, $\e(a^1_{21})=s_1t$. The augmentation on the Reeb chords of $\La_1\sharp_{m_(9_{46})}\La_2$ not shown in Figure~\ref{fig:doublesharp} can be identified with augmentations $\e_1$ of $\La_1$ and $\e_2$ of $\La_2$ such that $\e_1 (y_1)=-s_2^{-1}$ and $\e_2(y_2)=0$. In particular, we set $t^1_1= s_2^{-1} s_1 t$, $t^2_1=1$, $t^1_2= s_1^{-1}$, and $t^2_2=s_2$ so that both $s_1,$ and $s_2$ are determined by $\A ug^{\e(y_2)=0}(\La_2)$ and $\A ug^{\e(y_1)\neq 0}(\La_1)$. Then, we have that $$\A ug^{\e(b_1)=0}(\La_1\sharp_{m_(9_{46})}\La_2)\cong \A ug ^{\e (y_1)=-s_2^{-1}}(\La_1)\sqcup \A ug ^{\e (y_2)=0}(\La_2).$$

Similarly, if $\e(b_1)\neq 0$, we have that $\e(b_4)=0$, $\e(B^1_{11}):=s_1\neq 0$, and $\e(b_3)$ and $\e(b_5)$ can be set equal to $0$ with a dg-homotopy. Then we get $\e(b_1)=1, \e(b_2)=-1, \e(b_6)=-s_1^{-1}, \e(y_2)=s_1, \e(y_1)=0,$ Then, from $\partial a_{11}^2$ we know that $\e(B^1_{22}):=s_2\neq 0$, and  $\e(a^2_{21})=-s_2^{-1}$. Additionally, a dg-homotopy sets $\e(B^2_{11}), \e(B^2_{22})$ equal to $0$. The augmentation on the Reeb chords of $\La_1\sharp_{m_(9_{46})}\La_2$ not shown in Figure~\ref{fig:doublesharp} can be identified with augmentations $\e_1$ of $\La_1$ and $\e_2$ of $\La_2$ such that $\e_1(y_1)=0$ and $\e_2(y_2)\neq 0$.  We achieve this by setting $t^2_2= -s_2^{-1}$ and $t^2_1=1$. Then, we have the homeomorphism 

$$\A ug^{\e(b_1)\neq 0}(\La_1 \ \sharp_{m(9_{46})} \ \La_2)\cong \A ug^{\e(y_1)= 0}(\La_1)\times \A ug^{\e(y_2)\neq 0}(\La_2).$$ 

    \end{proof}

Next, we compute the augmentation variety of $\La_X \ \natural \ \La_Y$ for a restricted class of Legendrians.
\begin{figure}[!ht]
\begin{minipage}{2.5in}
    \centering
\includegraphics[width=0.9\linewidth]{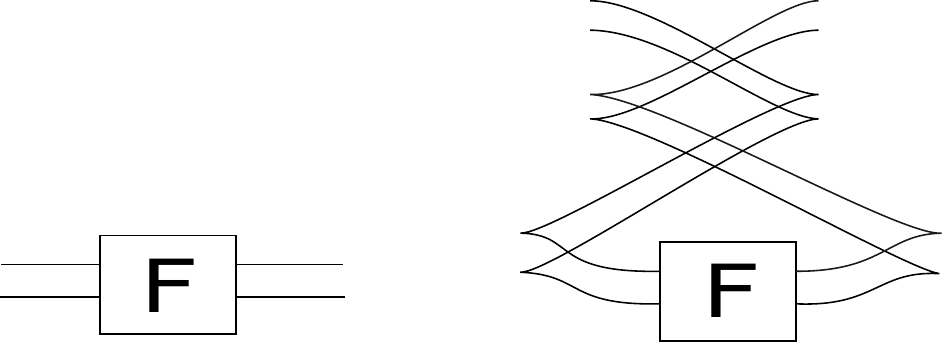}
    \caption{Satellite isotopy.}
    \label{fig:sat}
    \end{minipage}
    \begin{minipage}{3.5in}
    \centering
    \labellist
    \tiny{
    \pinlabel $y_1$ at  130 164
    \pinlabel $y_2$ at 180 158
    \pinlabel $y$ at 180 95
    \pinlabel $b_4$ at 13 91
    \pinlabel $b_7$ at 13 65
    }
    \endlabellist
\includegraphics[width=0.9\linewidth]{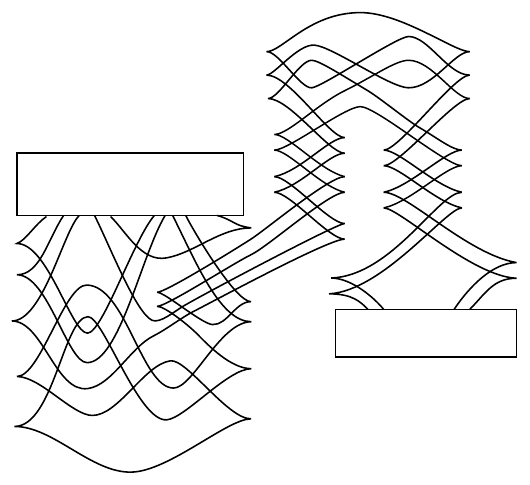}
    \caption{The front of $\La_X   \ \natural \   \La_Y$ after Legendrian isotopy.}
    \label{fig:doubleRI}
    \end{minipage}
\end{figure}

\begin{proposition}\label{prop:aug_natural}
Let $\La_X$ and $\La_Y$ be X-type and Y-type fronts respectively. We further assume that, for $\alpha=X,Y$, the augmentation variety $\A ug(\La_\alpha)$ admits a decomposition  
$$\A ug(\La_\alpha)=\A ug'(\La_\alpha)\sqcup pt,$$
where the point corresponds to the ruling of $\La_\alpha$ that contributes a constant term $1$ to the ruling polynomial $R_{\La_\alpha}(z)$.
Then,
    $$\A ug(\La_X\  \natural \ \La_Y) \cong \A ug'(\La_X)\sqcup \A ug'(\La_Y) \sqcup pt.$$
\end{proposition}

\begin{proof}

We compute the augmentation variety of the isotopic front of $\La_X  \ \natural \  \La_Y$ shown in Figure~\ref{fig:doubleRI} (see Remark~\ref{rem:XYsum}). By assumption, $\A ug^{\e(b_4)=0}(\La_X)$ is a point and $\A ug^{\e(b_1)=0}(\La_X)\cong\A ug'(\La_X)$. By Lemma~\ref{lem:stable} $\A ug (S(\La_X))\cong \A ug(\La_X)\sqcup pt$, where $\e_{12}$, as labeled in Equation~\eqref{eqn:stab}, corresponds to the isolated point.
If $\e(y_1)\neq 0$ on $S(\La_X)$ (with the additional Reidemeister I moves), then $\e$ is $\e_{12}$ from Equation ~\eqref{eqn:stab}, and therefore corresponds to a point in the augmentation variety of $S(\La_X)$. So $\A ug^{\e(y_1)\neq 0}S(\La_X)$ is a point and $\A ug^{\e(y_1)=0}S(\La_X)\cong \A ug(\La_X)$. For $\La_Y$ (with the additional Reidemeister I moves), the condition $\e(y_2)\neq 0$ is equivalent to $\e(y)\neq 0$, and any such augmentation corresponds to a point in $\A ug(\La_Y)$. Hence, $\Aug^{\e(y)\neq 0}(\La_Y)=pt$ and $\Aug^{\e(y)=0}(\La_Y)\cong \A ug'(\La_Y)$. Finally, by Proposition~\ref{prop:doublesharp}, 
 \begin{align*}\A ug(\La_X   \ \natural \   \La_Y)&\cong\A ug(S(\La_X) \ \sharp_{m(9_{46})}\  \La_Y)\\
 &\cong \A ug^{\e(y_1)\neq 0}(S(\La_X))\times \A ug^{\e(y_2)=0}(\La_Y)  \sqcup \A ug^{\e(y_1)=0}(S(\La_X))\times \A ug^{\e(y_2)\neq 0}(\La_Y)\\
 &\cong pt \times \A ug'(\La_Y) \sqcup \A ug(\La_X)\times pt\\
&\cong \A ug'(\La_Y)\sqcup \A ug'(\La_X)\sqcup pt
 \end{align*}
    
\end{proof}

\noindent We end this section by proving Theorem~\ref{thm:aug}.

\begin{proof}[Proof of Theorem~\ref{thm:aug}]
By definition $\A ug(\La_1\sqcup \La_2)\cong \A ug(\La_1)\times \A ug(\La_2)$. From Example~\ref{ex:7_4} we know that $\A ug(\La_{z^2})\cong (\C^*)^2$. As in the proof of Theorem~\ref{thm:main}, it suffices to consider the case when the constant term of $P(z^2)$ is non-zero. For $P(z^2)=\displaystyle{\sum_{k=0}^N a_k z^{2n_k}}$, recall that
$$\La_{P(z^2)}=\La_{a_0}\  \natural  \ \La_{a_1z^{2n_1}+1}\  \natural  \ \La_{a_2z^{2n_2}+1}\  \natural \cdots \natural\  \La_{a_Nz^{2n_N}+1}.$$
Recall also that $\La_k$ and $\La_{kz^{2n}+1}$ satisfy the conditions of Proposition~\ref{prop:aug_natural} for $n,k\geq 1$. By Proposition~\ref{prop:aug_natural}, Corollary~\ref{cor:aug_stack_n} and Proposition\ref{prop:kz2n+1}, 

\begin{align*}\A ug(\La_{P(z^2))}&\cong \sqcup_{k=1}^N \A ug (\La_{a_kz^{2n_k}+1})\sqcup \A ug(\La_{a_0-1})  \sqcup pt\\
& \cong \sqcup_{k=0}^N(\sqcup_{j=1}^{a_k}(\C^*)^{2n_k})\sqcup (\sqcup_{j=1}^{a_0-1} pt) \sqcup pt
\\
& \cong \sqcup_{k=0}^N(\sqcup_{j=1}^{a_k}(\C^*)^{2n_k})\sqcup (\sqcup_{j=1}^{a_0} pt).
\end{align*}

In particular, $\A ug(\La_{P(z^2)})$ is a disjoint union of $P(1)$ components. Each monomial $a_kz^{2n_k}$ of $P(z^2)$ corresponds to $\sqcup_{j=1}^{a_k}(\C^*)^{2n_k}$ in the augmentation variety $\A ug(\La_{P(z^2)})$. Each connected component has a single cluster chart and thus the augmentation variety has a trivial cluster structure.
\end{proof}

\section{Smoothly non-isotopic fillings}\label{sec:exotic}

In this section we prove Theorems~\ref{intro:prop:notsmooth},~\ref{thm:diff} and~\ref{thm:notsmooth}. We describe a strategy for detecting when two smooth immersed fillings $L_1, L_2$ are not smoothly isotopic, generalizing~\cite[Proposition 3.1]{Li_Tange_2020}. Instead of studying the pair of fillings $L_1, L_2\subset \mathbb{B}^4$ directly, we follow the common strategy of studying the $4$-manifolds with boundary given by their exteriors $W_1:=\mathbb{B}^4\backslash \nu(L_1)$ and $W_2:=\mathbb{B}^4\backslash \nu(L_2)$. If two fillings are smoothly isotopic rel. boundary then one can argue that their exteriors are diffeomorphic, as well as any two $4$-manifolds $X_1, X_2$ given by further attaching identical $2$-handles to both exteriors. We obstruct the $4$-manifolds by showing that one admits a Weinstein structure and the other does not.

See~\cite{gompf19994} for a general introduction to $4$-manifolds and Kirby calculus. A smooth immersed ribbon filling of a knot $\Lambda$ with transverse double points is given by a series of band moves and crossing changes on $\La$ that turn it into an unlink. Each component of the unlink is a carved $1$-handle. Every band move is encoded by adding a $0$-framed $2$-handle as shown in Figure~\ref{fig:ribbon_exterior}, and every crossing change is encoded by adding a $0$-framed $2$-handle as shown in Figure~\ref{fig:ribbon_exterior_double_point}. Thus, for an immersed ribbon filling $L$, the exterior $W_{L}$ has a Kirby diagram given by a link of unknots $K_{L}=(\cup_i H_i)\cup (\cup_j N_j)$ composed of an unlink of carved 1-handles $\cup_iH_i$ and a union of $0$-framed 2-handles $\cup_j N_j$. See~\cite[Section 6.2]{gompf19994} for more details and examples of exteriors of ribbon fillings.

\begin{figure}[H]
 \begin{tikzpicture}[scale=1]
\node[inner sep=0] at (0,0){\includegraphics[width=0.5\textwidth]{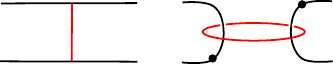}};
\node at (2.5,-0.5) {\textcolor{red}{$0$}};
\end{tikzpicture}
\caption{On the left is a band move, on the right is the corresponding handlebody diagram of the exterior of the ribbon surface.}
\label{fig:ribbon_exterior}
\end{figure}

\begin{figure}[H]
\centering
 \begin{tikzpicture}[scale=1]
\node[inner sep=0] at (0,0){\includegraphics[scale=1]{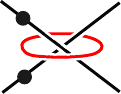}};
\node at (1,-0.2) {\textcolor{red}{$0$}};
\end{tikzpicture}
\caption{The exterior of an immersed surface near a double point obtained by a crossing change.}
\label{fig:ribbon_exterior_double_point}

\end{figure}

We need to compare pairs of exteriors with additional unknots added to both of them. In order to do so, we first define ``comparable'' Kirby diagrams of the exteriors. We invite the reader to reference examples shown in Figure~\ref{fig:m821} and Figure~\ref{fig:U_compatible} when reading Definition~\ref{defn:comparable} and Definition~\ref{defn:compatible}.

\begin{definition}\label{defn:comparable}
Let $L_1$ and $L_2$ be two smooth immersed ribbon fillings with transverse double points such that $\partial L_1=\partial L_2$. Suppose there exist Kirby diagrams $K^\ell=(\cup_{i=1}^n H^\ell_i)\cup (\cup_{j=1}^m N^\ell_j)\subset \R^3$ of $W_{\ell}$ for $\ell=1,2$, and a division $\R^3=S\cup T$ satisfying the following conditions:
\begin{enumerate}
    \item $S$ is the union of finitely many open sets containing all of the $2$-handles in either $K^1$ or $K^2$;
    \item $H^\ell_i\cap T\neq \emptyset$ for all $1\leq i\leq n$ and $\ell=1,2$;
    \item as sets $\sqcup_{i=1}^n H^1_i\cap T=\sqcup_{i=1}^n H^2_i\cap T$.
\end{enumerate}
Then, we say $W_1$ and $W_2$ have \textbf{comparable Kirby diagrams}. 

\end{definition}

\begin{definition}\label{defn:compatible}
Let $L_1$ and $L_2$ be two smooth immersed ribbon fillings with transverse double points such that $\partial L_1=\partial L_2$. Suppose the exteriors $W_1, W_2$ of $L_1,L_2$ have comparable Kirby diagrams $K^1,K^2\subset \R^3=T\cup S$ as above and contain $n>1$ carved $1$-handles. Suppose also that $K^1$ contains at least one unknotted $0$-framed $2$-handle.

A \textbf{$U$-compatible system} on two such comparable Kirby diagrams $K^1$ and $K^2$ is a disjoint union of unknots $U_1, \ldots, U_{n}$ contained in the interior of $T$ satisfying the following conditions:

\begin{enumerate}
    \item at least one unknot $U_i$ has linking number $0$ with a single carved $1$-handle and can be isotoped in $W_1$ so that it is adjacent to a $0$-framed $2$-handle as shown in Figure~\ref{fig:local_1};
    \item for a fixed order on the carved $1$-handles of $W_2$, $H^2_{i_1}, \ldots, H^2_{i_n}$, the unknots $U_j$ intersect only $H^2_{i_j}$ and $H^2_{i_{j+1}}$ in $K^2$ for $1\leq j\leq n-1$ with linking number $\pm1$,
    \item the unknot $U_n$ 
    intersects only $H^2_{i_{1}}$ in $K^2$ with $lk(H^2_{i_1},U_n)=\pm1$.
\end{enumerate}
If there exists a $U$-compatible system of unknots $U_1, \ldots, U_{n}$ on comparable Kirby diagrams of $W_1$ and $W_2$, then we say the exteriors of $L_1$ and $L_2$ are $U$-compatible.
\end{definition}

\begin{figure}[H]
\centering
\begin{tikzpicture}[scale=1]
\node[inner sep=0] at (0,0){\includegraphics[width=0.45\textwidth]{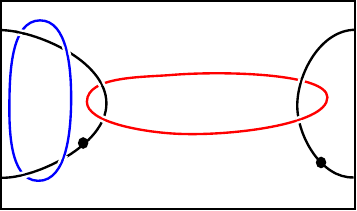}};
\node at (-1.8,1.4) {\textcolor{blue}{$U_1$}};
\node at (2,-0.8) {\textcolor{red}{$0$}};
\end{tikzpicture}
\caption{An exterior $W_1$ of a ribbon filling with an additional $(-1)$-framed $2$-handle $U_1$.}
\label{fig:local_1}
\end{figure}

\begin{figure}[!ht]
    \centering
 \begin{tikzpicture}[scale=1.5]
\node[inner sep=0] at (0,0){\includegraphics[width=0.5\textwidth]{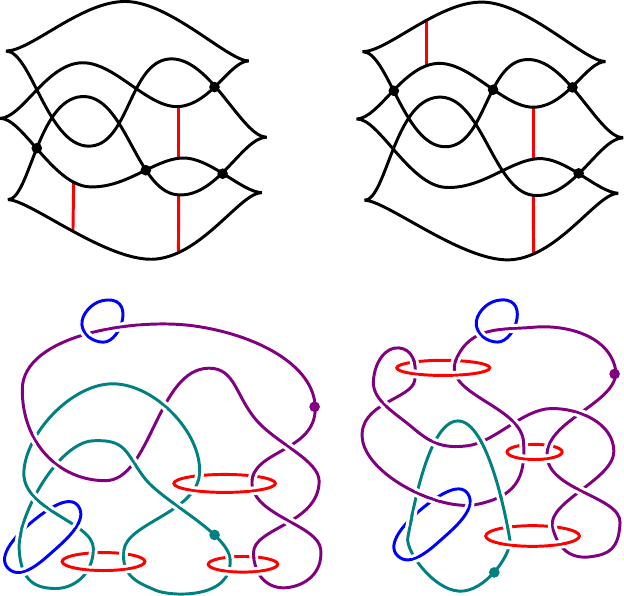}};
\node at (1,-0.8) {\textcolor{red}{$0$}};
\node at (0.5,-2.3) {\textcolor{blue}{$U_1$}};
\node at (1.8,-2.3) {\textcolor{red}{$0$}};
\node at (1.9,-1.5) {\textcolor{red}{$0$}};
\node at (-0.8,-1.8) {\textcolor{red}{$0$}};
\node at (-1.8,-2.5) {\textcolor{red}{$0$}};
\node at (-2.8,-2.3) {\textcolor{blue}{$U_1$}};
\node at (-0.6,-2.5) {\textcolor{red}{$0$}};
\node at (-0.6,-2.5) {\textcolor{red}{$0$}};
\node at (-2.2,0) {\textcolor{blue}{$U_2$}};
\node at (1.2,0) {\textcolor{blue}{$U_2$}};
\end{tikzpicture}
    \caption{Two fillings of $m(8_{21})$ that are not smoothly isotopic with their respective rulings and $U$-compatible exteriors shown below. }
    \label{fig:m821}
\end{figure}
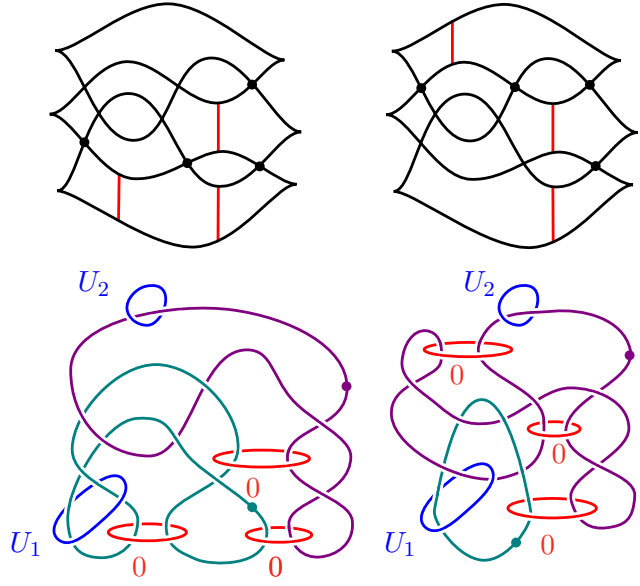

\begin{example}\label{ex:m821}
 The naive augmentation variety quotiented by dg-homotopy of the Legendrian knot $m(8_{21})$ consists of two irreducible components that intersect each other. There are two embedded fillings known to belong to the two different irreducible components and which are shown in the top row of Figure~\ref{fig:m821}. Two comparable Kirby diagrams of the exteriors are shown in the bottom row with a $U$-comparable system of unknots $U_1,U_2$ in blue. Since the exteriors of the two fillings are $U$-compatible, Theorem~\ref{intro:prop:notsmooth} implies the fillings are not smoothly isotopic. For another example of two fillings of a Legendrian ($\La_3$) with $U$-compatible exteriors, see Figure~\ref{fig:U_compatible}. 
\end{example}

\begin{figure}[!ht]
    \centering
 \begin{tikzpicture}[scale=1]
\node[inner sep=0] at (0,0){\includegraphics[width=0.75\textwidth]{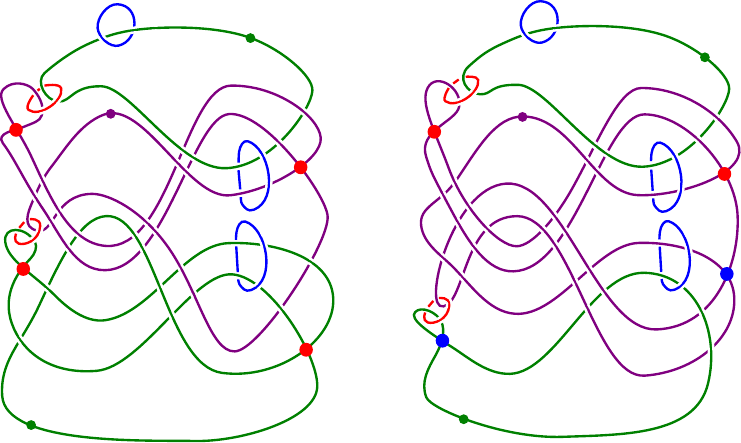}};
\node at (-0.6,2.6) {\textcolor{darkgreen}{$H'_1$}};
\node at (-0.3,0) {\textcolor{purple}{$H'_2$}};
\node at (-0.4,-2) {\textcolor{darkgreen}{$H'_3$}};
\node at (6,2.6) {\textcolor{darkgreen}{$H_1$}};
\node at (6.2,0) {\textcolor{purple}{$H_2$}};
\node at (5.9,-2) {\textcolor{darkgreen}{$H_3$}};
\node at (-1.3,0.2) {\textcolor{blue}{$U_2$}};
\node at (-2,-1.4) {\textcolor{blue}{$U_1$}};
\node at (5.3,0.2) {\textcolor{blue}{$U_2$}};
\node at (4.5,-1.3) {\textcolor{blue}{$U_1$}};
\node at (-5.8,0.3) {\textcolor{red}{$N'_1$}};
\node at (-5.6,2.7) {\textcolor{red}{$N'_2$}};
\node at (0.7,-0.8) {\textcolor{red}{$N_1$}};
\node at (1,2.7) {\textcolor{red}{$N_2$}};
\node at (-4.7,3) {\textcolor{blue}{$U_3$}};
\node at (2,3.2) {\textcolor{blue}{$U_3$}};
\end{tikzpicture}
    \caption{Two fillings $L_{1},L_{2}$ of $\La_3$ have $U$-compatible exteriors: pictured are the Kirby diagrams of the exteriors $W_{1},W_{2}$ and the unknots $U_1,U_2,U_3$ in blue. An isotopy takes $U_1$ in $W_{1}$ to the desired position adjacent to $N_1'.$ The blue and red dots denote the normal ruling corresponding to the filling.}
    \label{fig:U_compatible}
\end{figure}
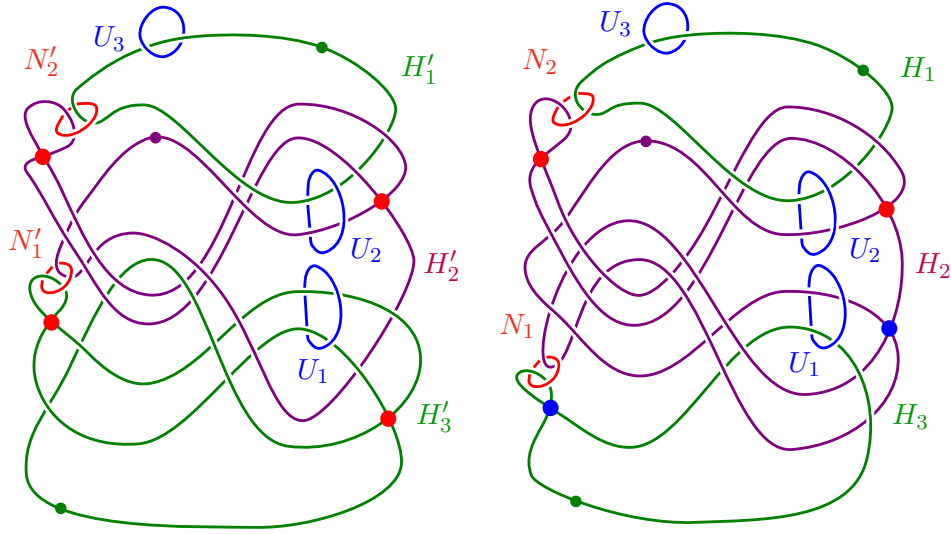

We show that immersed ribbon fillings $L_1,L_2$ of $\La$ with $U$-comparable exteriors are not smoothly isotopic.

\begin{proof}[Proof of Theorem~\ref{intro:prop:notsmooth}]

Suppose, for the sake of a contradiction, that $L_1$ and $L_2$ are smoothly isotopic, where the isotopy is the identity on the boundary. Note that $\partial W_i= \partial(B^4\backslash \nu(L_i))=(S^3\backslash \nu(\La))\sqcup (L_i\times S^1)$. The isotopy between $L_1$ and $L_2$ induces a diffeomorphism $\varphi$ between $W_1$ and $W_2$. Since the isotopy between $L_1$ and $L_2$ is the identity on $S^3,$ we have that $\varphi$ is the identity map on $S^3\backslash \nu(\La)$. Fix two comparable Kirby diagrams of $W_1$ and $W_2$, and let $U_1, \ldots, U_n$ be the set of unknots from the $U$-compatible system on $W_1, W_2$. Attach to both $W_1$ and $W_2$ $(-1)$-framed $2$-handles along the unknots $U_1, \ldots, U_{n-1}$, as well as a $(-2)$-framed $2$-handle $U_n$. By construction the diffeomorphism $\varphi$ extends to $W_1\cup_{i=1}^{n} U_i$ and $W_2 \cup_{i=1}^{n} U_i$.

We now show that $W_1\cup_{i=1}^{n} U_i$ does not admit a Weinstein structure. As shown in Figure~\ref{fig:local_not_stein}, there exists a sequence of handle-slides and isotopies to a handlebody diagram that contains an isolated $(-1)$ framed $2$-handle. By~\cite{Lisca_Matic_1997}, there are no smoothly embedded square $(-1)$ spheres in any Weinstein 4-manifold. Therefore, $W_1\cup_{i=1}^{n} U_i$ does not admit a Weinstein structure. 

\begin{figure}[!ht]
\centering
\begin{tikzpicture}[scale=1.5]
\node[inner sep=0] at (0,0){\includegraphics[width=0.8\textwidth]{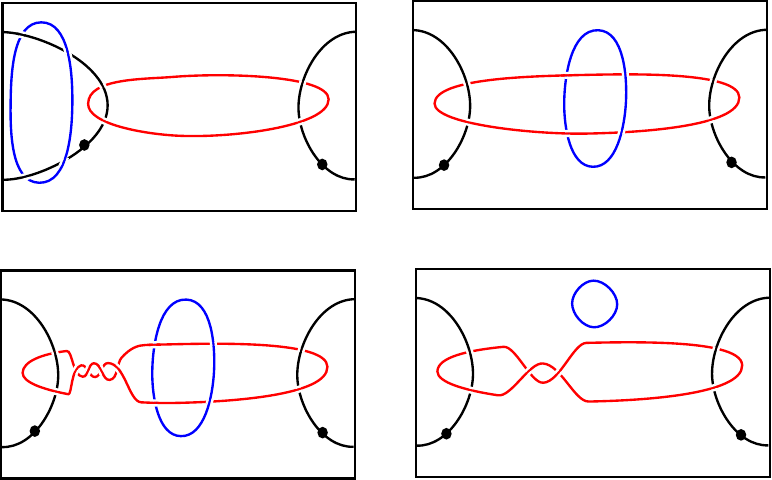}};
\node at (-1.5,1) {\textcolor{red}{$0$}};
\node at (-3.2,2.2) {\textcolor{blue}{$-1$}};
\node at (0,1.5) {$\rightarrow$};
\node at (0,1.8) {$(A)$};
\node at (2.8,1) {\textcolor{red}{$0$}};
\node at (2.8,-1.9) {\textcolor{red}{$-2$}};
\node at (-1.5,-1.9) {\textcolor{red}{$0$}};
\node at (2.8,2.2) {\textcolor{blue}{$-1$}};
\node at (-1.7,-0.8) {\textcolor{blue}{$-1$}};
\node at (2.7,-0.5) {\textcolor{blue}{$-1$}};
\node at (0,0) {$\swarrow$};
\node at (0,0.3) {$(B)$};
\node at (0,-1.5) {$\rightarrow$};
\node at (0,-1.2) {$(C)$};
\end{tikzpicture}
\caption{A series of isotopies and handleslides that produce a $4$-manifold without a Weinstein structure: (A) isotope the $(-1)$-framed 2-handle off of the carved 1-handle; (B) add canceling twists to the $0$-framed $2$-handle; (C) a blow-down handleslide.}
\label{fig:local_not_stein}
\end{figure}
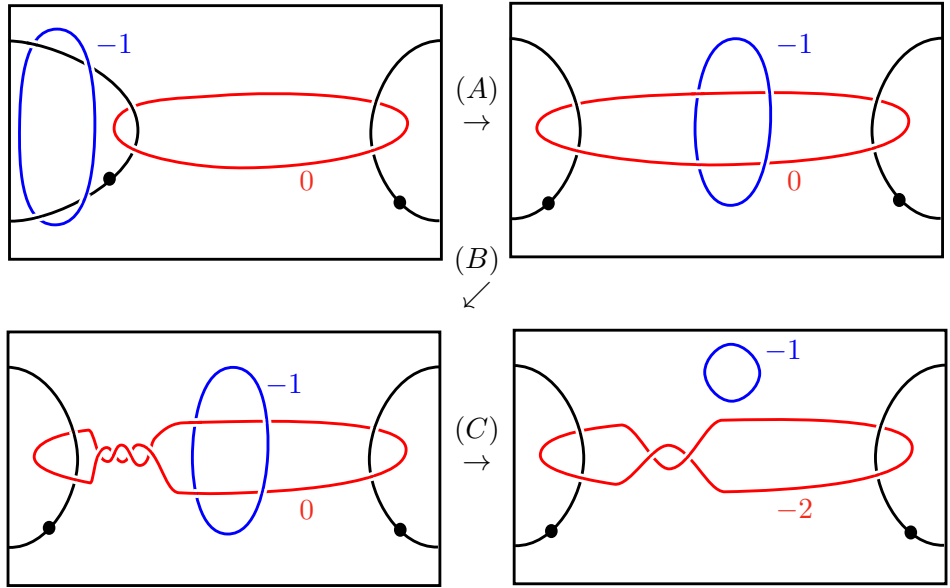

Next, we show that $W_2\cup_{i=1}^{n} U_i$ does admit a Weinstein structure. We rely on the topological characterization of Weinstein $4$-manifolds of Eliashberg~\cite{eliashberg} and Gompf~\cite[Theorem 1.3, Proposition 2.3]{gompf19994}. Namely, a smooth, compact, connected, oriented $4$-manifold $X$ admits a Weinstein structure (inducing the given orientation) if and only if it can be presented
as a handlebody by attaching $2$-handles to a framed link in $\partial(\mathbb{D}^4\cup $1$-\text{handles})=\sharp^m(S^1\times S^2)$, where the link is drawn in a Legendrian standard form and the framing coefficient on each link component $\La$ is given by $tb(\La)-1$. Let $W_{2}$ have a Kirby diagram given by $K^2=(\cup_{i=1}^n H^2_i)\cup (\cup_{j=1}^m N^2_j)$ where $H_i^2$ are carved $1$-handles, and $N_j^2$ are $0$-framed $2$-handles corresponding to band moves and double points of $L_2$ respectively. We show that through a sequence of handle slides and isotopies, there is another Kirby diagram of $W_2$ such that every carved $1$-handle is an unknot with no crossings, and every $2$-handle $\La$ is in Legendrian standard form with framing $tb(\La)-1$.

We begin by eliminating all but one of the carved $1$-handles in $K^2$ with isotopies and handle slides. Consider the graph $\Gamma$ whose vertices are the carved $1$-handles and edges between two vertices $H^2_i$ and $H^2_j$ are given by any $(-1)$-framed $2$-handle $U_k$ that links $H^2_i$ and $H^2_j$. By construction, $\Gamma$ is a linear graph with $n$ vertices. Label the edges $e_1, \ldots, e_{n-1}$ corresponding to the unknots $U_1, \ldots, U_{n-1}$. Label the vertices $v_1, \ldots, v_n$ corresponding to the carved $1$-handles $H^2_{i_1}, \ldots, H^2_{i_{n}}$. Starting with $H^2_{i_1}$, handleslide it over $H^2_{i_2}$ so it is no longer linked with $U_{1}$. By abuse of notation, let $H^2_{i_1}$ denote the carved $1$-handle after a handleslide. Observe that $H^2_{i_1}$ is the connect sum of $H^2_{i_1}$ and $H^2_{i_2}$, a disjoint pair of unknots. Therefore $H^2_{i_1}$ is still an unknot. Moreover, any $2$-handle that intersected $H^2_{i_2}$ with the exception of $U_{1}$ now also intersects $H^2_{i_1}$. By construction, $U_{1}$ and $H^2_{i_2}$ are now in canceling position, so we cancel them. The graph associated to the new handlebody diagram is the contraction of $\Gamma$ along the first edge $e_1$. We continue this process contracting the edges of $\Gamma$ until we have only one vertex $v_1$ left. We now have a single carved $1$-handle $H^2_{i_{1}}$ that is topologically an unknot with $0$-framed $2$-handles $N_1, \ldots, N_m$ and a $(-2)$-framed $2$-handle $U_{n}$. Perform smooth isotopies until the carved $1$-handle is an unknot with no crossings at the cost of adding crossings to $N_1, \ldots, N_m, U_n$. 

We end by showing that with additional handle-slides, every $2$-handle $\La$ can be drawn in Legendrian standard form with the correct framing $tb(\La)-1$. First, since $U_n$ has framing $-2$ it has Legendrian representative with $tb$ equal to $-1$ respectively. To ensure that each $N_j$ has a Legendrian representative for $1\leq j\leq m$, perform an arbitrary number of handleslides of $N_j$ along $U_{n}$ to lower the framing arbitrarily on $N_j$. Recall that a handleslide of a $2$-handle $K_1$ with framing $f_1$ along a $2$-handle $K_2$ with framing $f_2$ results in the $2$-handle $K_1'$ given by the band connect sum of $K_1$ and $K_2$ with framing $f_1+f_2\pm 2lk(K_1,K_2)$. We choose a band connect sum so that we have framing $f_1+f_2-|2lk(K_1,K_2)|$ which is guaranteed to be negative since $N_j$ and $U_{n}$ initially have zero and negative framings respectively. Every smooth knot has a Legendrian representative with bounded max-tb, so for an arbitrarily negative framing each $2$-handle has a Legendrian representative with framing $tb-1.$ So, every $2$-handle now has a Legendrian representative. Thus, by the topological characterization of Weinstein $4$-manifolds we can conclude that $W_2\cup_{i=1}^{n} U_i$ admits a Weinstein structure.

Since $W_1\cup_{i=1}^{n}U_i$ does not admit a Weinstein structure and $W_2\cup_{i=1}^{n}U_i$ does, we have arrived at a contradiction. Therefore, $L_1, L_2$ are not smoothly isotopic.
\end{proof}

\begin{remark}
In comparison to the argument in Li-Tange~\cite{Li_Tange_2020}, we introduce an additional $(-2)$-framed $2$-handle $U_n$ to avoid the need for explicit computations and to ensure that we can arbitrarily lower the framing on the $2$-handles as needed.
\end{remark}

We apply Theorem~\ref{intro:prop:notsmooth} to prove Theorem~\ref{thm:diff}. In order to do so, we first associate exact Lagrangian fillings to rulings of the Legendrians $\La_k$.

\begin{definition}\label{defn:fillings_n}
For $1\leq i\leq k-1$ let $\rho_{i}$ denote the ruling of $\La_k$ with $i-1$ total layers with blue type switches. Next, let $L_i$ denote the embedded exact Lagrangian filling of  $\La_k$ given by the pinching sequence compatible with the ruling $\rho_i$ where there is a pinch move for every leftmost switch as shown in Figure~\ref{fig:pinch}.
\end{definition}

\begin{remark}
It is straightforward to show that $L_i$ induces the augmentation $\e_i$ that maps the Reeb chord associated to the switched crossing to $1$. For example, for $m(9_{46})$ there are two disk fillings $L_1$ and $L_2$ corresponding to the normal rulings with red and blue switches. $L_1$ is the exact Lagrangian filling induced by the augmentation where $\e(b_1)=1$.
\end{remark}

\begin{figure}[ht!]
    \centering
\includegraphics[width=0.1\linewidth]{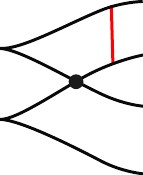}
    \caption{Pinch moves compatible with a ruling.}
    \label{fig:pinch}
\end{figure}

\begin{proposition}\label{prop:lan_not_smooth}
The exact disk Lagrangian fillings $L_i$ and $L_j$ of $\La_{k}$ are not smoothly isotopic for $i\neq j$.
\end{proposition}

\begin{proof}
By construction, the exteriors of $L_i, L_j$ have comparable Kirby diagrams. By Theorem~\ref{intro:prop:notsmooth} it suffices to show that the exteriors $W_i$ and $W_j$ are U-compatible for $i\neq j$. The fillings $L_i$ and $L_j$ are constructed with $k-1$ pinch moves and $k$ unknots, so that their exteriors have $k\geq 1$ carved 1-handles. By definition, every filling $L_i$ corresponds to a normal ruling $\rho_i$ of $\La_{k}$ where $1\leq i\leq k-1$.  Without loss of generality, assume that $\rho_i$ has more red type switches than $\rho_j.$ Recall that the layers of $\La_{k}$ are labeled $1,\ldots, k-1$ from  top to bottom. Let $1\leq k_1<k-1$ denote the top layer of $\La_{k}$ where $\rho_j$ has blue type switches (and where, by assumption $\rho_i$ has red type switches).

\begin{figure}[!ht]
    \centering
 \begin{tikzpicture}[scale=1.5]
\node[inner sep=0] at (0,0){\includegraphics[width=0.8\textwidth]{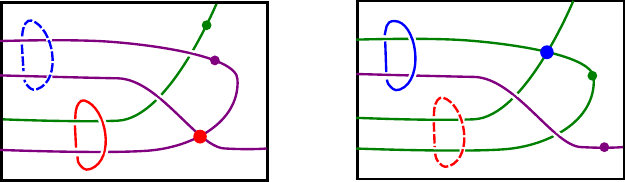}};
\node at (-3.3,0.9) {\textcolor{blue}{$U_B$}};
\node at (-2.6,-0.1) {\textcolor{red}{$U_R$}};
\node at (1.5,0.9) {\textcolor{blue}{$U_B$}};
\node at (2.2,-0.1) {\textcolor{red}{$U_R$}};
\end{tikzpicture}
    \caption{Possible Unknots $U_B,U_R$ for a U-compatible system on exteriors of $L_i$. The two distinct carved $1$-handles depend on the normal ruling.}
    \label{fig:compatible}
\end{figure}
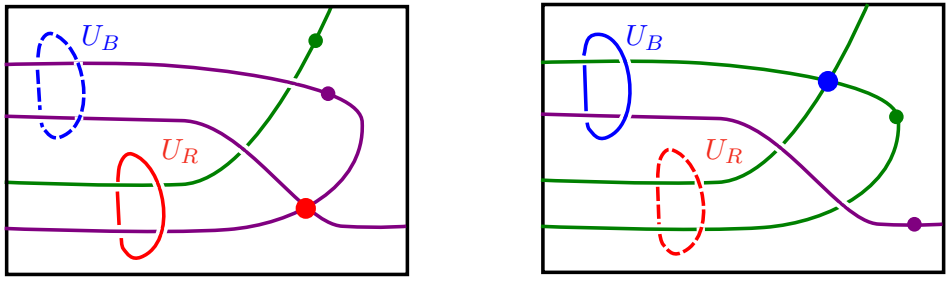

 We now describe how to choose unknots $U$ that form a $U$-compatible system on $W_i,W_j$. Decorate the Kirby diagrams of the exteriors $W_i, W_j$ with the rulings $\rho_i, \rho_j.$ Focus on the decorated Kirby diagram of $W_j$, and consider only the rightmost switches. Near each rightmost switch of $\rho_j$ there are two candidate unknots: $U_B$ and $U_R$ as shown in Figure~\ref{fig:compatible}. Add the unknot $U_B$ to the $U$-compatible system set if the switch is of blue type, and otherwise add the unknot $U_R$. Both $U_B$ and $U_R$ intersect only two distinct carved $1$-handles in $W_j$. There are $k-1$ rightmost switches, so we have added $k-1$ unknots $U_1, \ldots, U_{k-1}$. We add the final unknot $U_{k}$ to our $U$-compatible system by choosing an unknot with linking number 1 with only the topmost carved $1$-handle in both exteriors. Observe that for $W_i$ the $k_1$-th switch of $\rho_i$ is of red type and differs from that of $\rho_j$ which is of blue type.
 The unknot we chose for that layer was $U_B$, and it intersects one carved $1$-handle with linking number $0$ in $W_i$. It is straightforward to verify that there is an isotopy pushing this unknot $U_B$ towards the leftmost red switch of $\rho_i$ so that it has the same configuration as the one shown in Figure~\ref{fig:local_1}. We then have a $U$-compatible system of unknots for comparable Kirby diagrams of $W_i,W_j.$
\end{proof}

\begin{remark}
    The fillings $L_i, L_j$ of $\La_k$ need not be topologically isotopic. However, some pairs have exteriors that are related by a diffeomorphism given by an involution and are therefore topologically isotopic. In particular, for $\La_k$ there are at least $\lfloor \frac{k}{2}\rfloor$ such pairs.
\end{remark}

We end this section by proving Theorem~\ref{thm:diff}, and Theorem~\ref{thm:notsmooth}

\begin{proof}[Proof of Theorem~\ref{thm:diff}]
Recall that $\La^g$ is the checkerboard Legendrian whose front is shown in Figure~\ref{fig:checkerboard}.
If $\chi(L)=1-2g$, for $g\geq 0$, consider the disjoint union $\La_k\sqcup\La^g$ where $\La_k$ is to the left of $\La^g$ (if you want a connected link, then take a connect sum $\La_k\#\La^g$ joining the rightmost cusp of $\La_k$ to the leftmost cusp of $\La^g)$. Note that $\La^g$ has a unique graded ruling corresponding to a filling $L^g$ of genus $g$, while $\La_k$ has at least $k$ distinct disk fillings $L_i$, $1\leq i\leq k$. We consider the two genus $g$ fillings of $\La_k\#\La^g$ given by $L_i\#L^g$ and $L_j\#L^g$  for $i\neq j$.
The exteriors of both fillings have comparable Kirby diagrams and we can construct a $U$-compatible system on the exteriors that is identical on the exteriors of $\La^g$, and so that for the exteriors of $L_i, L_j$ it is the $U$-compatible system from the proof of Proposition~\ref{prop:lan_not_smooth}. Then by Theorem~\ref{intro:prop:notsmooth} $L_i\#L^g$ and $L_j\#L^g$ are smoothly non-isotopic. For $\chi(L)=-2g$ for $g\geq 0$, we consider $\La_k\#\La^g\#\La_{m(3_1)}$. An analogous argument to the case of orientable fillings allows us to conclude that $\La_k\#\La^g\#\La_{m(3_1)}$ has $k$ distinct fillings with $\chi(L)=-2g$ that are pairwise smoothly non-isotopic. 
\end{proof}

\begin{proof}[Proof of Theorem~\ref{thm:notsmooth}]
Consider first $P(z^2)=kz^{2n}+1$. Then, we construct $k$ immersed fillings with $n$ double points and one disk filling. The disk filling of $\La_{kz^{2n}+1}$ is constructed by the unique pinching sequence on the subfronts $\La_{k+1}$
and $m(9_{46})$ defined by the normal ruling which results in an unlink that we fill in with disks. The $k$ immersed fillings are constructed by first pinching the subfront $\La_{k+1}$ in $\La_{kz^{2n}+1}$. Next pinch the $n$ subfronts $m(9_{46})$, following the normal rulings. Then, one obtains an ungraded $T(-2,n)$ torus link which we fill with an immersed disk filling with $n$ double points of index $+1$. Any pair of the $k$ immersed fillings differ only in the $\La_{k+1}$ subfront, therefore one can construct a $U$-compatible system on Kirby diagrams of the exteriors of the fillings following the proof of Proposition~\ref{prop:lan_not_smooth}. By Theorem~\ref{intro:prop:notsmooth} these pairs of fillings are not smoothly isotopic.

It suffices to end by considering the case of $P(z^2)=a_{n_1}z^{2n_1}+a_{n_2}z^{2n_2}+1$ for $n_1,n_2\geq 0$, where $\La_{P(z^2)}:=\La_{a_{n_1}z^{2n_1}+1} \ \natural \  \La_{a_{n_2}z^{2n_2}+1}$. Then, as discussed in Proposition~\ref{prop:together}, there are three families of rulings corresponding to $a_{n_1}$ immersed fillings with $n_1$ double points, $a_{n_2}$ immersed fillings with $n_2$ double points, and a disk filling. The $a_{n_1}$ immersed fillings are constructed by different pinching sequences on the $\La_{a_{n_1+1}}$ subfront of the $\La_{a_{n_1}z^{2n_1}+1}$ factor. For the $\La_{a_{n_2}z^{2n_2}+1}$ factor, the pinching sequence compatible with the ruling yields an unlink which we fill with disks. Thus, it is straightforward to construct a $U$-compatible system on any pair of the exteriors of the $a_{n_1}$ immersed fillings following the proof of Proposition~\ref{prop:lan_not_smooth}. By Theorem~\ref{intro:prop:notsmooth} these pairs of fillings are not smoothly isotopic. The same argument holds for pairs of $a_{n_2}$ immersed fillings.

\end{proof}

\bibliographystyle{alpha}
\bibliography{OP}

@article{Cao_Gallup_Hayden_Sabloff_2014, title={Topologically distinct {L}agrangian and symplectic fillings}, volume={21}, DOI={10.4310/mrl.2014.v21.n1.a7}, number={1}, journal={Mathematical Research Letters}, author={Cao, Chang and Gallup, Nathaniel and Hayden, Kyle and Sabloff, Joshua M.}, year={2014}, pages={85–99}}

@article{FI,
	author = {Fuchs, Dmitry and Ishkhanov, Tigran},
	doi = {10.17323/1609-4514-2004-4-3-707-717},
	fjournal = {Moscow Mathematical Journal},
	issn = {1609-3321},
	journal = {Moscow Mathematical Journal},
	mrclass = {57R17 (53D40 57M27)},
	mrnumber = {2119145},
	mrreviewer = {Lenhard L. Ng},
	number = {3},
	pages = {707--717, 783},
	title = {Invariants of {L}egendrian knots and decompositions of front diagrams},
	url = {https://doi.org/10.17323/1609-4514-2004-4-3-707-717},
	volume = {4},
	year = {2004}}

@article{Ng_2003, title={Computable {L}egendrian invariants}, volume={42}, DOI={10.1016/s0040-9383(02)00010-1}, number={1}, journal={Topology}, author={Ng, Lenhard L.}, year={2003},  pages={55–82}}

@article{NRSS,
	author = {Ng, Lenhard and Rutherford, Dan and Shende, Vivek and Sivek, Steven},
	doi = {10.4310/MRL.2017.v24.n6.a14},
	fjournal = {Mathematical Research Letters},
	issn = {1073-2780},
	journal = {Mathematical Research Letters},
	mrclass = {57M25 (18F99 57R17)},
	mrnumber = {3762698},
	mrreviewer = {Tye Lidman},
	number = {6},
	pages = {1845--1874},
	title = {The cardinality of the augmentation category of a {L}egendrian link},
	url = {https://doi.org/10.4310/MRL.2017.v24.n6.a14},
	volume = {24},
	year = {2017}}

@article {CN,
       AUTHOR = {Casals, Roger and Ng, Lenhard},
     TITLE = {Braid loops with infinite monodromy on the {L}egendrian
              contact {DGA}},
   JOURNAL = {Journal of Topology},
  FJOURNAL = {Journal of Topology},
    VOLUME = {15},
      YEAR = {2022},
    NUMBER = {4},
     PAGES = {1927--2016},
      ISSN = {1753-8416,1753-8424},
   MRCLASS = {53D10 (57K10 57K33)},
  MRNUMBER = {4584583},
MRREVIEWER = {Yu\ Pan},
       DOI = {10.1112/topo.12264},
       URL = {https://cfcix1b13095ec5284139s5ox9c9wbxxuq6xnxfayc.eds.tju.edu.cn/10.1112/topo.12264},
}

@article{CG,
	author = {Casals, Roger and Gao, Honghao},
	doi = {10.4007/annals.2022.195.1.3},
	journal = {Annals of Mathematics},
	number = {1},
	title = {Infinitely many {L}agrangian fillings},
	volume = {195},
	year = {2022}}

@article{Gao2, title={Augmentations, fillings, and clusters}, volume={34}, DOI={10.1007/s00039-024-00673-y}, number={3}, journal={Geometric and Functional Analysis}, author={Gao, Honghao and Shen, Linhui and Weng, Daping}, year={2024},  pages={798–867}}

@article {CGGS,
    AUTHOR = {Casals, Roger and Gorsky, Eugene and Gorsky, Mikhail and
              Simental, Jos\'e},
     TITLE = {Algebraic weaves and braid varieties},
   JOURNAL = {American Journal of Mathematics},
  FJOURNAL = {American Journal of Mathematics},
    VOLUME = {146},
      YEAR = {2024},
    NUMBER = {6},
     PAGES = {1469--1576},
      ISSN = {0002-9327,1080-6377},
   MRCLASS = {14R25 (57K10 57K33)},
  MRNUMBER = {4855860},
}

@article {CGGILS,
    AUTHOR = {Casals, Roger and Gorsky, Eugene and Gorsky, Mikhail and Le,
              Ian and Shen, Linhui and Simental, Jos\'e},
     TITLE = {Cluster structures on braid varieties},
   JOURNAL = {Journal of the American Mathematical Society},
  FJOURNAL = {Journal of the American Mathematical Society},
    VOLUME = {38},
      YEAR = {2025},
    NUMBER = {2},
     PAGES = {369--479},
      ISSN = {0894-0347,1088-6834},
   MRCLASS = {13F60 (14M15 20F36)},
  MRNUMBER = {4868947},
       DOI = {10.1090/jams/1048},
       URL = {https://cfcix1b13095ec5284139s5ox9c9wbxxuq6xnxfayc.eds.tju.edu.cn/10.1090/jams/1048},
}

@article {Ru,
    AUTHOR = {Rutherford, Dan},
     TITLE = {Thurston-{B}ennequin number, {K}auffman polynomial, and ruling
              invariants of a {L}egendrian link: the {F}uchs conjecture and
              beyond},
   JOURNAL = {International Mathematics Research Notices},
  FJOURNAL = {International Mathematics Research Notices},
      YEAR = {2006},
     PAGES = {Art. ID 78591, 15},
      ISSN = {1073-7928,1687-0247},
   MRCLASS = {57M27 (57R17)},
  MRNUMBER = {2219227},
MRREVIEWER = {Justin\ Sawon},
       DOI = {10.1155/IMRN/2006/78591},
       URL = {https://cfcix1b13095ec5284139sfvvkfbf0fwwc69qxfayc.eds.tju.edu.cn/10.1155/IMRN/2006/78591},
}

@article{Etnyre_Ng_Vertesi_2013,
title={Legendrian and {t}ransverse twist knots},
volume={15},
DOI={10.4171/jems/383},
number={3},
journal={Journal of the European Mathematical Society},
author={Etnyre, John B. and Ng, Lenhard L. and Vértesi, Vera}, year={2013},
pages={969–995}}

@article{Gao_Rutherford_2021, title={Non-fillable augmentations of {T}wist Knots}, 
volume={2023}, 
DOI={10.1093/imrn/rnab224},
number={2},
journal={International Mathematics Research Notices},
author={Gao, Honghao and Rutherford, Dan},
year={2021},
pages={1255–1291}}

@article{HR,
	author = {Henry, Michael B. and Rutherford, Dan},
	doi = {10.1112/jtopol/jtu013},
	fjournal = {Journal of Topology},
	issn = {1753-8416},
	journal = {Journal of Topology},
	mrclass = {57R17 (53D42 57M27)},
	mrnumber = {3335247},
	mrreviewer = {Steven Sivek},
	number = {1},
	pages = {1--37},
	title = {Ruling polynomials and augmentations over finite fields},
	url = {https://doi.org/10.1112/jtopol/jtu013},
	volume = {8},
	year = {2015}}

@article{Baykur_Horn-Morris_2016, title={Fillings of genus–1 open books and 4–braids}, DOI={10.1093/imrn/rnw281}, journal={International Mathematics Research Notices}, author={Baykur, R. İnanç and Horn-Morris, Jeremy Van}, year={2016}}

@article{Boileau_Orevkov_2001, title={Quasi-positivité d’une courbe analytique dans une boule pseudo-convexe}, volume={332}, DOI={10.1016/s0764-4442(01)01945-0}, number={9}, journal={Comptes Rendus de l’Académie des Sciences - Series I - Mathematics}, author={Boileau, Michel and Orevkov, Stepan}, year={2001},  pages={825–830}}

@article {Sab,
    AUTHOR = {Sabloff, Joshua M.},
     TITLE = {Augmentations and rulings of {L}egendrian knots},
   JOURNAL = {International Mathematics Research Notices},
  FJOURNAL = {International Mathematics Research Notices},
      YEAR = {2005},
    NUMBER = {19},
     PAGES = {1157--1180},
      ISSN = {1073-7928,1687-0247},
   MRCLASS = {57M25 (53D40 57R17)},
  MRNUMBER = {2147057},
MRREVIEWER = {Quach thi C\^am V\^an},
       DOI = {10.1155/IMRN.2005.1157},
       URL = {https://cfcix1b13095ec5284139sfvvkfbf0fwwc69qxfayc.eds.tju.edu.cn/10.1155/IMRN.2005.1157},
}

@article{Pan_Rutherford_2023, title={Augmentations and immersed {L}agrangian fillings}, volume={16}, DOI={10.1112/topo.12280}, number={1}, journal={Journal of Topology}, author={Pan, Yu and Rutherford, Dan}, year={2023}, pages={368–429}}

@article{Pan_Rutherford_2021, title={Functorial lch for immersed {L}agrangian cobordisms}, volume={19}, DOI={10.4310/jsg.2021.v19.n3.a5}, number={3}, journal={Journal of Symplectic Geometry}, author={Pan, Yu and Rutherford, Dan}, year={2021}, pages={635–722}}

@article{Oba_2020, title={Surfaces in the $4$-disk with the same boundary and Fundamental Group}, volume={27}, DOI={10.4310/mrl.2020.v27.n1.a13}, number={1}, journal={Mathematical Research Letters}, author={Oba, Takahiro}, year={2020}, pages={265–279}}

@article{Li_Tange_2020, title={Smoothly non-isotopic {L}agrangian disk fillings of {L}egendrian knots}, volume={213}, DOI={10.1007/s10711-020-00575-x}, number={1}, journal={Geometriae Dedicata}, author={Li, Youlin and Tange, Motoo}, year={2020}, pages={211–225}}

@article{polterovich, title={The surgery of {L}agrange submanifolds}, volume={1}, DOI={10.1007/bf01896378}, number={2}, journal={Geometric and Functional Analysis}, author={Polterovich, Leonid}, year={1991}, pages={198–210}}

@article{eliashberg, 
title={Topological Characterization Of {S}tein Manifolds Of Dimension $>2$}, 
volume={01}, 
DOI={10.1142/s0129167x90000034}, 
number={01}, 
journal={International Journal of Mathematics}, 
author={Eliashberg, Yakov}, 
year={1990},
 pages={29–46}}

@book{gompf19994,
  title={4-manifolds and {K}irby calculus},
  author={Gompf, Robert E and Stipsicz, Andr{\'a}s I},
  volume={20},
  year={1999},
  publisher={American Mathematical Society},
  address={Providence, R.I.},
  series={Graduate Studies in Mathematics},
  isbn={978-0821809945}
}

@article{Lisca_Matic_1997, title={Tight contact structures and Seiberg-Witten invariants}, volume={129}, DOI={10.1007/s002220050171}, number={3}, journal={Inventiones Mathematicae}, author={Lisca, P. and Matić, G.}, year={1997},  pages={509–525}}

@article{ChNg,
	author = {Chongchitmate, Wutichai and Ng, Lenhard},
	doi = {10.1080/10586458.2013.750221},
	fjournal = {Experimental Mathematics},
	issn = {1058-6458},
	journal = {Exp. Math.},
	mrclass = {57M25 (53Cxx 57R17)},
	mrnumber = {3038780},
	number = {1},
	pages = {26--37},
	title = {An atlas of {L}egendrian knots},
	url = {https://doi.org/10.1080/10586458.2013.750221},
	volume = {22},
	year = {2013}}

@article{JCHLLW,
    author = {Asplund, Johan and Capovilla-Searle, Orsola and Hughes,James and Leverson, Caitlin and  Li, Wenyuan and Wu, Angela},
    title = {Decompositions of augmentation varieties via weaves and rulings},
    journal = {arXiv preprint, arXiv:2508.20226 },
    year = {2025}}

@article{Hayden_Kim_Miller_Park_Sundberg_2025, title={Seifert surfaces in the 4-ball}, DOI={10.4171/jems/1703}, journal={Journal of the European Mathematical Society}, author={Hayden, Kyle and Kim, Seungwon and Miller, Maggie and Park, JungHwan and Sundberg, Isaac}, year={2025}}

@article{Akbulut_1991, title={A solution to a conjecture of {Z}eeman}, volume={30}, DOI={10.1016/0040-9383(91)90028-3}, number={3}, journal={Topology}, author={Akbulut, S.}, year={1991}, pages={513–515}}

@article{Leverson_2016, title={Augmentations and rulings of {L}egendrian knots}, volume={14}, DOI={10.4310/jsg.2016.v14.n4.a5}, number={4}, journal={Journal of Symplectic Geometry}, author={Leverson, Caitlin}, year={2016}, pages={1089–1143}}

@article {CLLMPT,
    AUTHOR = {Capovilla-Searle, Orsola and Legout, No\'emie and Limouzineau,
              Ma\"ylis and Murphy, Emmy and Pan, Yu and Traynor, Lisa},
     TITLE = {Obstructions to reversing {L}agrangian surgery in {L}agrangian
              fillings},
   JOURNAL = {Journal of Symplectic Geometry},
  FJOURNAL = {The Journal of Symplectic Geometry},
    VOLUME = {22},
      YEAR = {2024},
    NUMBER = {3},
     PAGES = {599--672},
      ISSN = {1527-5256,1540-2347},
   MRCLASS = {57K43 (53D12 57K33)},
  MRNUMBER = {4819508},
DOI = {10.4310/jsg.241001214710},
       URL = {https://cfcix1b13095ec5284139s9npvqvfvwk0u6kovfayc.eds.tju.edu.cn/10.4310/jsg.241001214710},
}

@article{CSHW,
title={Augmentations, fillings and clusters for 2-bridge links},
author={Capovilla-Searle, Orsola and Hughes, James and Weng, Daping},
year={2025},
archivePrefix={arXiv:2308.11858},
primaryClass={math.SG},
journal={Selecta Mathematica},
volume={31},
number={102},
}

@article{Casals_weng, 
title={Microlocal theory of {L}egendrian links and cluster algebras},
volume={28}, 
DOI={10.2140/gt.2024.28.901}, 
number={2}, 
journal={Geometry $\&$ Topology}, 
author={Casals, Roger and Weng, Daping}, 
year={2024}, 
pages={901–1000}}

@article{Casals_zaslow, 
 title={{L}egendrian weaves: {N}–Graph Calculus, flag moduli and applications}, 
 volume={26}, 
 DOI={10.2140/gt.2022.26.3589}, 
 number={8}, 
 journal={Geometry $\&$ Topology}, 
 author={Casals, Roger and Zaslow, Eric}, 
 year={2022}, 
 pages={3589–3745}}

@article{eliashberg_polterovich, 
title={Local {L}agrangian 2-Knots are Trivial}, 
volume={144}, DOI={10.2307/2118583}, 
number={1}, journal={The Annals of Mathematics},
 author={Eliashberg, Yakov and Polterovich, Leonid}, 
year={1996}, pages={61},
}

@article {STWZ,
	AUTHOR = {Shende, Vivek and Treumann, David and Williams, Harold and
	Zaslow, Eric},
	TITLE = {Cluster varieties from {L}egendrian knots},
	JOURNAL = {Duke Math. J.},
	FJOURNAL = {Duke Mathematical Journal},
	VOLUME = {168},
	YEAR = {2019},
	NUMBER = {15},
	PAGES = {2801--2871},
	ISSN = {0012-7094},
	MRCLASS = {53D12 (05E99 32S60 53D30)},
	MRNUMBER = {4017516},
	DOI = {10.1215/00127094-2019-0027},
	URL = {https://doi.org/10.1215/00127094-2019-0027},
}

@article{Arnold_1990, 
title={Singularities of caustics and wave fronts}, 
DOI={10.1007/978-94-011-3330-2}, 
journal={Mathematics and Its Applications}, 
author={Arnold, V. I.}, year={1990}}

@book{Geiges_2008,
place={Cambridge},
title={An introduction to contact topology},
publisher={Cambridge University Press},
author={Geiges, Hansj\"org}, year={2008}}

@article{EN_survey,
author={Etnyre, John B. and Ng, Lenhard L.},
title={Legendrian contact homology in $\mathbb{R}^3$.},
journal={In Surveys in differential
geometry 2020. Surveys in 3-manifold topology and geometry},
volume={25},
pages={103–161.},
publisher={Int. Press, Boston, MA,},
year={2022}}

@article {Fuchs03,
    AUTHOR = {Fuchs, Dmitry},
     TITLE = {Chekanov--{E}liashberg invariant of {L}egendrian knots:
              existence of augmentations},
   JOURNAL = {Journal of Geometry and Physics},
  FJOURNAL = {Journal of Geometry and Physics},
    VOLUME = {47},
      YEAR = {2003},
    NUMBER = {1},
     PAGES = {43--65},
      ISSN = {0393-0440,1879-1662},
   MRCLASS = {57M25 (57M50 57R17)},
  MRNUMBER = {1985483},
MRREVIEWER = {Vladimir\ V.\ Tchernov},
       DOI = {10.1016/S0393-0440(01)00013-4},
       URL = {https://doi.org/10.1016/S0393-0440(01)00013-4},
}

@article {Chekanov,
    AUTHOR = {Chekanov, Yuri},
     TITLE = {Differential algebra of {L}egendrian links},
   JOURNAL = {Inventiones Mathematicae},
  FJOURNAL = {Inventiones Mathematicae},
    VOLUME = {150},
      YEAR = {2002},
    NUMBER = {3},
     PAGES = {441--483},
      ISSN = {0020-9910},
   MRCLASS = {53D35 (57M27 57R17)},
  MRNUMBER = {1946550},
MRREVIEWER = {John B. Etnyre},
       DOI = {10.1007/s002220200212},
       URL = {https://doi.org/10.1007/s002220200212},
}

@article{SFT, title={Introduction to {S}ymplectic {F}ield {T}heory},
DOI={10.1007/978-3-0346-0425-3_4},
journal={Visions in Mathematics},
author={Eliashberg, Yakov and Givental, Alexander and Hofer, Helmut},
year={2000},
pages={560–673}}

@article {ChekanovPushkar,
    AUTHOR = {Pushkar , P. E. and Chekanov, Yu. V.},
     TITLE = {Combinatorics of fronts of {L}egendrian links, and
              {A}rnold's 4-conjectures},
   JOURNAL = {Uspekhi Matematicheskikh Nauk},
  FJOURNAL = {Uspekhi Matematicheskikh Nauk},
    VOLUME = {60},
      YEAR = {2005},
    NUMBER = {1(361)},
     PAGES = {99--154},
      ISSN = {0042-1316,2305-2872},
   MRCLASS = {58K15 (53D10 53D12 57M25)},
  MRNUMBER = {2145660},
MRREVIEWER = {John\ B.\ Etnyre},
       DOI = {10.1070/RM2005v060n01ABEH000808},
       URL = {https://doi.org/10.1070/RM2005v060n01ABEH000808},
}

@article{survey_wiscon, title={Constructions of {L}agrangian cobordisms}, DOI={10.1007/978-3-030-80979-9_5}, journal={Association for Women in Mathematics Series}, author={Blackwell, Sarah and Legout, Noémie and Leverson, Caitlin and Limouzineau, Maÿlis and Myer, Ziva and Pan, Yu and Pezzimenti, Samantha and Suárez, Lara Simone and Traynor, Lisa}, year={2021}, pages={245–272}}

@article{Pan_2017, title={Exact {L}agrangian fillings of {L}egendrian $(2,n)$ torus links}, volume={289}, DOI={10.2140/pjm.2017.289.417}, number={2}, journal={Pacific Journal of Mathematics}, author={Pan, Yu}, year={2017},  pages={417–441}}

@article{Chantraine_2010, title={Lagrangian concordance of {L}egendrian knots}, volume={10}, DOI={10.2140/agt.2010.10.63}, number={1}, journal={Algebraic $\&$ Geometric Topology}, author={Chantraine, Baptiste}, year={2010},  pages={63–85}}

@article{Hayden_Sundberg_2024, title={Khovanov homology and exotic surfaces in the 4-ball}, volume={0}, DOI={10.1515/crelle-2024-0001}, number={0}, journal={Journal für die reine und angewandte Mathematik (Crelles Journal)}, author={Hayden, Kyle and Sundberg, Isaac}, year={2024}}

@article{Juhasz_Miller_Zemke_2021, title={Transverse invariants and exotic surfaces in the 4–ball}, volume={25}, DOI={10.2140/gt.2021.25.2963}, number={6}, journal={Geometry $\&$ Topology}, author={Juhász, András and Miller, Maggie and Zemke, Ian}, year={2021}, pages={2963–3012}}

@article{Hamming_Gao,
title={Isomorphism in the augmentation category},
author={Gao, Honghao and Liu, Hamming},
year={2026},
journal={arXiv preprint, 	arXiv:2602.12555}}

@article{Galashin_Lam_Sherman-Bennett_2025, title={Braid variety cluster structures, {I}{I}: General type}, volume={243}, DOI={10.1007/s00222-025-01390-5}, number={3}, journal={Inventiones mathematicae}, author={Galashin, Pavel and Lam, Thomas and Sherman-Bennett, Melissa}, year={2025}, pages={1079–1127}}

@article{Shende_Treumann_Zaslow_2016, title={{L}egendrian knots and constructible sheaves}, volume={207}, DOI={10.1007/s00222-016-0681-5}, number={3}, journal={Inventiones mathematicae}, author={Shende, Vivek and Treumann, David and Zaslow, Eric}, year={2016}, pages={1031–1133}}

@article{Traynor_2001, title={Generating function polynomials for {L}egendrian links}, volume={5}, DOI={10.2140/gt.2001.5.719}, number={2}, journal={Geometry $\&$ Topology}, author={Traynor, Lisa}, year={2001},  pages={719–760}}

@article{Fuchs_Rutherford_2011, title={Generating families and {L}egendrian contact homology in the standard contact space}, volume={4}, DOI={10.1112/jtopol/jtq033}, number={1}, journal={Journal of Topology}, author={Fuchs, Dmitry and Rutherford, Dan}, year={2011}, pages={190–226}}

@article{Juhasz_Zemke_2019, title={Distinguishing slice disks using knot {F}loer homology}, volume={26}, DOI={10.1007/s00029-019-0531-6}, number={1}, journal={Selecta Mathematica}, author={Juhász, András and Zemke, Ian}, year={2019}}

@article{EHK,
	Author = {Ekholm, Tobias and Honda, Ko and K{\'a}lm{\'a}n, Tam{\'a}s},
	Doi = {10.4171/JEMS/650},
	Fjournal = {Journal of the European Mathematical Society (JEMS)},
	Issn = {1435-9855},
	Journal = {Journal of the European Mathematical Society},
	Mrclass = {Preliminary Data},
	Mrnumber = {3562353},
	Number = {11},
	Pages = {2627--2689},
	Title = {Legendrian knots and exact {L}agrangian cobordisms},
	Url = {http://dx.doi.org/10.4171/JEMS/650},
	Volume = {18},
	Year = {2016}}

@article{SuXiYu,
	author = {Su, Tao and Xie, Baiting and Yu, Chenglong},
	journal = {arXiv preprint arXiv:2603.27888v1},
	title = {Log-concavity from enumerative geometry of planar curve singularities},
	year = {2026}}

@article{Hayden,
title={Exotically knotted disks and complex curves},
author={Kyle Hayden},
year={2021},
journal={arXiv preprint, arXiv:2003.13681 }}

@article{Ian,
title={Bar-Natan skein lasagna modules and exotic surfaces in 4-manifolds},
author={Ian Sullivan},
year={2025},
journal={arXiv preprint, arXiv:2504.03968}}

@article{MR26,
title={Undecidability problems for semifree DG algebras},
author={Manolescu, Ciprian and Rozenblyum, Nick},
journal={arXiv preprint, arXiv:2605.08122v1},
year={2026}
}

\end{document}